\documentclass[a4paper,12pt,reqno]{amsart}
\usepackage{amssymb}
\usepackage[margin=1in]{geometry}
\usepackage{enumitem}
\usepackage{graphicx,color}

\usepackage[bookmarksdepth=2]{hyperref}

\newcommand{\revisionone}[1]{{#1}}
\newcommand{\revisiontwo}[1]{{#1}}
\newcommand{\revisionfinal}[1]{{#1}}
\newcommand{\ignore}[1]{}

\numberwithin{equation}{section}

\theoremstyle{plain}
\newtheorem{theorem}{Theorem}[section]
\newtheorem*{theorem*}{Theorem}
\newtheorem{lemma}[theorem]{Lemma}
\newtheorem{corollary}[theorem]{Corollary}
\newtheorem{proposition}[theorem]{Proposition}

\theoremstyle{definition}

\newtheorem{example}[theorem]{Example}
\newtheorem{remark}[theorem]{Remark}

\theoremstyle{remark}

\DeclareMathOperator{\pv}{pv}

\newcommand{\D}{d}
\newcommand{\pr}{\mathbf{P}}
\newcommand{\ex}{\mathbf{E}}

\newcommand{\R}{\mathbf{R}}
\newcommand{\fourier}{\mathcal{F}}

\newcommand{\ind}{\mathbf{1}}
\newcommand{\sub}{\subseteq}
\newcommand{\ph}{\varphi}

\newcommand{\eps}{\varepsilon}

\newcommand{\tnorm}[1]{\| #1 \|}
\newcommand{\scalar}[1]{\left\langle #1 \right\rangle}
\newcommand{\tscalar}[1]{\langle #1 \rangle}

\newcommand{\abs}[1]{\left| #1 \right|}
\newcommand{\tabs}[1]{\lvert #1 \rvert}
\newcommand{\expr}[1]{\left( #1 \right)}
\newcommand{\hl}{{(0, \infty)}}
\newcommand{\thet}{\vartheta}
\newcommand{\pvint}{\pv\!\!\int}
\newcommand{\gammau}{\revisionone{\Gamma}}
\newcommand{\gammal}{\revisionone{\gamma}}

\usepackage{ifthen}
\newcommand{\formula}[2][nolabel]%
{%
 \ifthenelse{\equal{#1}{nolabel}}%
 {\begin{align*} #2 \end{align*}}%
 {%
  \ifthenelse{\equal{#1}{}}%
  {\begin{align} #2 \end{align}}%
  {\begin{align} \label{#1} \begin{aligned} #2 \end{aligned} \end{align}}%
 }%
}

\begin{document}

%
%

\title[Hitting times of points]{Hitting times of points for symmetric L\'evy processes with completely monotone jumps}
\author{Tomasz Juszczyszyn, Mateusz Kwa{\'s}nicki}
\thanks{Work supported by NCN grant no. 2011/03/D/ST1/00311}
\thanks{Mateusz Kwa{\'s}nicki was supported by fellowship co-financed by European Union within European Social Fund}
\address{Tomasz Juszczyszyn, Mateusz Kwa{\'s}nicki \\ Institute of Mathematics and Computer Science \\ Wroc{\l}aw University of Technology \\ ul. Wybrze{\.z}e Wyspia{\'n}\-skiego 27 \\ 50-370 Wroc{\l}aw, Poland}
\email{mateusz.kwasnicki@pwr.edu.pl, tomek.juszczyszyn@gmail.com}

\begin{abstract}
Small-space and large-time estimates and asymptotic expansion of the distribution function and (the derivatives of) the density function of hitting times of points for symmetric L\'evy processes are studied. The L\'evy measure is assumed to have completely monotone density function, and a scaling-type condition $\inf \xi \Psi''(\xi) / \Psi'(\xi) > 0$ is imposed on the L\'evy--Khintchine exponent $\Psi$. Proofs are based on generalised eigenfunction expansion for processes killed upon hitting the origin.
\end{abstract}

\maketitle

%
%

\section{Introduction and statement of the results}
\label{sec:intro}

Let $X$ be a one-dimensional L\'evy process, that is, a real-valued stochastic process with stationary and independent increments, c\`adl\`ag paths, and initial value $X_0 = 0$. The process $X$ is completely characterised by its \revisiontwo{L\'evy--Khintchine exponent $\Psi$, which is given by the L\'evy--Khintchine formula:
\formula{
 \Psi(\xi) & = \revisionfinal{-\log(\ex e^{i \xi X_1})} 
 = a \xi^2 - i b \xi + \int_{\R \setminus \{0\}} (1 - e^{i \xi z} + i \xi z \ind_{(-1, 1)}(z)) \nu(\D z)
}
for $\xi \in \R$, where $a \ge 0$ is the Gaussian component, $b \in \R$ is the drift coefficient and $\nu$ is a non-negative measure such that $\int_{\R \setminus \{0\}} \min(1, z^2) \nu(dz) < \infty$, called L\'evy measure.}
 The \emph{first hitting time} of a point $x \in \R$ is defined by the formula
\formula{
 \tau_x = \inf \{ t \ge 0 : X_t = x \} .
}
In this article estimates and asymptotic formulae, in terms of the L\'evy--Khintchine exponent $\Psi$, for the tail and the density function of $\tau_x$ are derived, under a number of conditions on the process $X$.


The distribution of $\tau_x$ plays an important role in various contexts: local times and excursion theory (\cite{bib:bg64,bib:c10,bib:m76}), potential theory (\cite{bib:bg68}), penalisation problems (\cite{bib:yyy09a, bib:yyy09b}). The estimates of $\tau_x$ may also prove useful in the study of one-dimensional unimodal L\'evy processes, developed recently in~\cite{bib:bgr14a,bib:bgr14,bib:g14}. More precisely, description of $\tau_x$ is the limiting case of a more general problem of finding the time and place the process $X$ first hits a (small) ball, see~\cite{bib:bgr14}.

Surprisingly little is known about the properties of $\tau_x$ for general L\'evy processes. By~\cite[Theorem~43.3 and Remark~43.6]{bib:s99}, if $1 / \tabs{\Psi}$ is integrable at infinity, then
\formula[eq:general]{
 \int_{\R} e^{i \xi x} \ex e^{-\lambda \tau_x} \D x & = \frac{c_\lambda}{\lambda + \Psi(\xi)} \, , && \text{with} & c_\lambda & = \expr{\int_\R \frac{1}{\lambda + \Psi(\xi)} \, \D \xi}^{-1} .
}
The inversion of the Laplace and Fourier transforms in~\eqref{eq:general} is often problematic. An application of the inverse Fourier transform to both sides of~\eqref{eq:general} leads to an expression for $\ex e^{-\lambda \tau_x}$ in terms of an oscillatory integral. In fact,
\formula[eq:potential]{
 u_\lambda(x) & = c_\lambda^{-1} \ex e^{-\lambda \tau_x}
}
is a well-studied object, the $\lambda$-potential density of $X$. Nevertheless, a closed-form expression for $u_\lambda$ is known only in some special cases, e.g. when $X$ is stable and $\lambda = 0$, or when $X$ is relativistic with $\beta = 2$ (\revisiontwo{with the notation of Example~\ref{ex:tauxest:relativistic} below}
) and $\lambda = 1$. Therefore, in order to \revisiontwo{invert} 
the Laplace transform in~\eqref{eq:potential}, one needs additional regularity of $\Psi$. This is the rough idea of the proof of the main result of~\cite{bib:k12}, which is recalled as Theorem~\ref{th:taux} below, and which is the starting point for our development.

There are essentially two classes of L\'evy processes for which the description of $\tau_x$ simplifies dramatically and has been studied. When $X$ is an $\alpha$-stable process, $\tau_x$ is equal in distribution to $x^\alpha \tau_1$ (\emph{scaling}), so the originally two-dimensional problem becomes one-dimensional. Numerous results are available in this case. In particular, a complete series expansion of the distribution function of $\tau_x$ is known (see~\cite{bib:p08} for processes with one-sided jumps, \cite{bib:bgr61,bib:c10,bib:p67,bib:yyy09a} for the symmetric case, and~\cite{bib:kkpw14} for the general result). \revisionfinal{Other} 
closely related \revisiontwo{results} 
for the stable case (unimodality, distributional identities, applications) can be found in~\cite{bib:ls14,bib:s11,bib:yyy09b}.


The distribution of $\tau_x$ for $x > 0$ is rather well-studied also for L\'evy processes with negative jumps only (also known as spectrally negative processes). Then $\tau_x$ is equal to the first passage time through the level $x$, $\tau_x = \inf \{ t \ge 0 : X_t \ge x \}$, and fluctuation theory for L\'evy processes can be used to study the properties of $\tau_x$. We refer to~\cite[Chapter~9]{bib:s99} for more information.

For non-stable L\'evy processes with two-sided jumps, we are aware of no \revisiontwo{estimates or asymptotic formulae similar to the main results of this article}.


Throughout the article, $X$ is assumed to be symmetric, that is, $b = 0$ and $\nu(E) = \nu(-E)$ for all Borel $E \sub \R$. In this case $\Psi$ \revisiontwo{is a real function with non-negative values}
. We impose two additional restrictions: we require $X$ to have completely monotone jumps and satisfy a certain scaling-type condition. These notions are briefly discussed below.

Recall that a function \revisiontwo{$f : \hl \to \R$} 
is said to be \emph{completely monotone} if it is infinitely \revisiontwo{differentiable} 
and \revisionfinal{$(-1)^n f^{(n)}(\xi) \ge 0$} 
for all $\xi > 0$ and $n = 0, 1, 2, \dots$ By Bernstein's theorem, this is equivalent to \revisionfinal{$f$} 
being the Laplace transform of a non-negative Radon measure on $[0, \infty)$. \revisiontwo{Similarly, we say that a process} 
$X$ has \emph{completely monotone jumps} \revisiontwo{if its} 
L\'evy measure $\nu$ is absolutely continuous with respect to the Lebesgue measure, and its density is a completely monotone function on $(0, \infty)$. Note that due to symmetry, the density of $\nu$ on $(-\infty, 0)$ is \emph{absolutely monotone}: its derivatives of all orders are non-negative.

L\'evy processes with completely monotone jumps (without the symmetry condition) were introduced by Rogers in~\cite{bib:r83}, see also~\cite{bib:k15}. In the symmetric case, an equivalent condition can be given in terms of $\Psi$. Recall that $\psi$ is a complete Bernstein function if and only if
\formula{
 \psi(\xi) & = c_1 + c_2 \xi + \int_{(0, \infty)} \frac{\xi}{s + \xi} \, \frac{\mu(\D s)}{s}
}
for $\xi \ge 0$, where $c_1, c_2 \ge 0$ and $\mu$ is a non-negative measure for which the above integral converges (see~\cite{bib:ssv10}). A symmetric L\'evy process $X$ has completely monotone jumps if and only if $\Psi(\xi) = \psi(\xi^2)$ for a complete Bernstein function $\psi$ (see~\cite{bib:k11,bib:r83}). The most prominent examples of symmetric processes with completely monotone jumps are stable processes, with $\Psi(\xi) = c \tabs{\xi}^\alpha$ for some $c > 0$ and $\alpha \in (0, 2]$. This class includes also mixtures of stable processes and relativistic L\'evy processes (discussed later in this section), as well as variance gamma process and geometric stable processes (which with probability one do not hit single points and thus are not considered here\revisiontwo{; see~\cite{bib:ssv06} for definitions and properties of these processes}).

The aforementioned \emph{scaling-type condition} of order $\alpha$ requires that
\formula[eq:reg:1]{
 \frac{\xi \Psi''(\xi)}{\Psi'(\xi)} & \ge \alpha - 1
}
for all $\xi > 0$. Here $\alpha$ is an arbitrary real number, although in our main theorems we assume that $\alpha \in (1, 2]$. The scaling-type condition plays a crucial role in our development. By integration, \eqref{eq:reg:1} implies that (and in fact, it is equivalent to)
\formula{
 \frac{\Psi'(\xi_2)}{\Psi'(\xi_1)} & \ge \expr{\frac{\xi_2}{\xi_1}}^{\alpha - 1}
}
for all $\xi_2 > \xi_1 > 0$. In Lemma~\ref{l:ratioest} we will see that~\eqref{eq:reg:1} also gives (but it is essentially stronger than)
\formula[eq:reg:3]{
 \frac{\Psi(\xi_2)}{\Psi(\xi_1)} & \ge \expr{\frac{\xi_2}{\xi_1}}^\alpha
}
for all $\xi_2 > \xi_1 > 0$. This explains why we call~\eqref{eq:reg:1} a scaling-type condition.

We note that the scaling-type condition of order $\alpha > 1$ implies that $\pr(\tau_x < \infty) = 1$ for all $x \in \R$. Indeed, by~\eqref{eq:reg:3}, $1 / \tabs{\Psi}$ is not integrable near $0$, so $X$ is recurrent by Chung--Fuchs criterion (\cite[Theorem~37.5]{bib:s99}). Furthermore, again by~\eqref{eq:reg:3}, $1 / \tabs{\Psi}$ is integrable at infinity, so $\pr(\tau_x < \infty) > 0$ by~\cite[Remark~43.6]{bib:s99}. Now $\pr(\tau_x < \infty) = 1$ follows by~\cite[Remark~43.12]{bib:s99}.

The scaling-type condition~\eqref{eq:reg:1} with $\alpha \in (1, 2]$ is satisfied by the typical examples of symmetric L\'evy processes with completely monotone jumps which hit single points with probability $1$: stable, mixed stable (see Example~\ref{ex:tauxest:mixed}) and relativistic (see Example~\ref{ex:tauxest:relativistic}). An equivalent form of~\eqref{eq:reg:1}, as well as a sufficient condition in terms of the L\'evy measure, are given in Remark~\ref{rem:levy}. Nevertheless, \eqref{eq:reg:1} is rather restrictive, see Example~\ref{ex:counter}. We conjecture that the estimates of $\pr(\tau_x > t)$ hold in greater generality, for example, with~\eqref{eq:reg:1} replaced by $\Psi(\xi_2) / \Psi(\xi_1) \ge C (\xi_2 / \xi_1)^\alpha$ for some $C > 0$ and $\alpha > 1$ (a more general version of~\eqref{eq:reg:3}, see~\cite{bib:bgr14a,bib:bgr14,bib:g14}). However, with the present methods, we were unable to significantly relax the assumption that~\eqref{eq:reg:1} holds with $\alpha > 1$.

For symmetric processes with completely monotone jumps, $\Psi$ is an increasing function on $(0, \infty)$. Let $\Psi^{-1}$ denote the inverse function of the restriction of $\Psi$ to $(0, \infty)$. Our first main result provides large $t$ and small $x$ estimates of $\pr(\tau_x > t)$ and its time derivatives. A corollary that follows extends the estimate of $\pr(\tau_x > t)$ (with no time derivative) to the full range of $t > 0$ and $x \in \R \setminus \{0\}$. The constants in these estimates are given explicitly, see Remark~\ref{rem:const}.

\revisiontwo{Below we state the main results of the paper.}

\begin{theorem}
\label{th:tauxest:reg}
Suppose that $X$ is a symmetric L\'evy process with completely monotone jumps, which satisfies the scaling-type condition~\eqref{eq:reg:1} for some $\alpha \in (1, 2]$. Then there are positive constants $C_1(\alpha, n)$, $C_2(\alpha, n)$, $C_3(\alpha, n)$ such that
\formula[eq:tauxest:reg:2]{
 \frac{C_1(\alpha, n)}{t^{n+1} \tabs{x} \Psi(1/\tabs{x}) \Psi^{-1}(1/t)} & \le (-\tfrac{\D}{\D t})^n \pr(\tau_x > t) \le \frac{C_2(\alpha, n)}{t^{n+1} \tabs{x} \Psi(1/\tabs{x}) \Psi^{-1}(1/t)}
}
for all $n \ge 0$, $t > 0$ and $x \in \R \setminus \{0\}$ such that $t \Psi(1/\tabs{x}) \ge C_3(\alpha, n)$.
\end{theorem}

\begin{corollary}
\label{cor:nzero}
For $n = 0$, the conclusion of Theorem~\ref{th:tauxest:reg} can be rewritten as follows: there are positive constants $\tilde{C}_1(\alpha)$ and $\tilde{C}_2(\alpha)$ such that
\formula[eq:tauxest:reg:5]{
 \frac{\tilde{C}_1(\alpha)}{1 + t \tabs{x} \Psi(1/\tabs{x}) \Psi^{-1}(1/t)} & \le \pr(\tau_x > t) \le \frac{\tilde{C}_2(\alpha)}{1 + t \tabs{x} \Psi(1/\tabs{x}) \Psi^{-1}(1/t)}
}
for all $t > 0$ and $x \in \R \setminus \{0\}$.
\end{corollary}

Under an additional regularity condition, the above two-sided estimates can be turned into asymptotic formulae for $\pr(\tau_x > t)$ as $t \to \infty$ or $x \to 0$. Recall that a function $\psi : \hl \to \R$ is \emph{regularly varying at infinity} with index $\alpha$ if $\lim_{\xi \to \infty} \psi(k \xi) / \psi(\xi) = k^\alpha$ for all $k > 0$. If the same equation holds with the limit as $\xi \to 0^+$ instead of $\xi \to \infty$, $\psi$ is said to be \emph{regularly varying at zero} with index $\alpha$. \revisiontwo{Observe that if $\Psi$ satisfies the scaling-type condition~\eqref{eq:reg:1} and it is regularly varying with index $\gamma$ at infinity or at zero, then, by~\eqref{eq:reg:3}, we have $\gamma \ge \alpha$.}

\begin{theorem}
\label{th:tauxasymp:reg}
Suppose that $X$ is a symmetric L\'evy process with completely monotone jumps, which satisfies the scaling-type condition~\eqref{eq:reg:1} for some $\alpha \in (1, 2]$.
\begin{enumerate}[label={\rm (\alph*)}]
\item\label{th:tauxasymp:reg:a}
If $\Psi$ is regularly varying at infinity with index $\gamma \in (1, 2]$, then the limit
\formula{
 & \lim_{x \to 0} \expr{\tabs{x} \Psi(\tfrac{1}{\tabs{x}}) (-\tfrac{\D}{\D t})^n \pr(\tau_x > t)}
}
exists and belongs to $(0, \infty)$ for all $n \ge 0$ and $t > 0$.
\item\label{th:tauxasymp:reg:b}
If $\Psi$ is regularly varying at zero with index $\delta \in (1, 2]$, then the limit
\formula{
 \lim_{t \to \infty} \expr{t^{n + 1} \Psi^{-1}(\tfrac{1}{t}) (-\tfrac{\D}{\D t})^n \pr(\tau_x > t)}
}
exists and belongs to $(0, \infty)$ for all $n \ge 0$ and $x \in \R \setminus \{0\}$.
\end{enumerate}
\end{theorem}

The limits in the above theorem are given explicitly by rather complicated expressions, see Remark~\ref{rem:lim}. In the following examples, application of Theorems~\ref{th:tauxest:reg} and~\ref{th:tauxasymp:reg} to three types of symmetric L\'evy processes with completely monotone jumps is given. Technical details, such as verification of~\eqref{eq:reg:1}, are left to the reader.


Note that our main results for symmetric stable processes follow immediately from the full series expansion \revisiontwo{given in~\cite{bib:kkpw14}}: Theorem~\ref{th:tauxasymp:reg} gives the first term, and two-sided estimates of Theorem~\ref{th:tauxest:reg} follow easily by a scaling argument. On the other hand, Theorems~\ref{th:tauxest:reg} and~\ref{th:tauxasymp:reg} seem to be completely new for non-stable processes.


\begin{example}
\label{ex:tauxest:relativistic}
Suppose that $1 < \alpha < \beta \le 2$ and let $X$ be the L\'evy process with L\'evy--Khintchine exponent $\Psi(\xi) = (1 + \tabs{\xi}^\beta)^{\alpha/\beta} - 1$ (sometimes $X$ is called the \emph{relativistic L\'evy process}). Then
\formula{
 \frac{c_1(\alpha, n) \tabs{x}^{\alpha - 1} (1 + \tabs{x})^{\beta - \alpha}}{t^{n+1-1/\alpha} (1 + t)^{1/\alpha-1/\beta}} & \le \expr{-\frac{\D}{\D t}}^n \pr(\tau_x > t) \le \frac{c_2(\alpha, n) \tabs{x}^{\alpha - 1} (1 + \tabs{x})^{\beta - \alpha}}{t^{n+1-1/\alpha} (1 + t)^{1/\alpha-1/\beta}}
}
for all $n \ge 0$, $t > 0$ and $x \in \R \setminus \{0\}$ such that $t / \min(\tabs{x}^\alpha, \tabs{x}^\beta) \ge c_3(\alpha, n)$. Furthermore, finite and positive limits
\formula{
 & \lim_{x \to 0} \expr{\tabs{x}^{1 - \alpha} (-\tfrac{\D}{\D t})^n \pr(\tau_x > t)} , && \lim_{t \to \infty} \expr{t^{n + 1 - 1 / \beta} (-\tfrac{\D}{\D t})^n \pr(\tau_x > t)}
}
exist for all $n \ge 0$, $t > 0$ and $x \in \R \setminus \{0\}$. Note that the restriction $\alpha > 1$ is required by the scaling-type condition~\eqref{eq:reg:1}\revisiontwo{. Otherwise, if $\alpha \le 1$, we have that} 
$\pr(\tau_x < \infty) = 0$.
\end{example}

\begin{example}
\label{ex:tauxest:mixed}
Suppose that $1 < \alpha < \beta \le 2$, and let $X$ be the L\'evy process with L\'evy--Khintchine exponent $\Psi(\xi) = \tabs{\xi}^\alpha + \tabs{\xi}^\beta$ (that is, $X$ is the sum of independent stable L\'evy processes). Then
\formula{
 \frac{c_1(\alpha, n) \tabs{x}^{\beta - 1} (1 + t)^{1/\alpha - 1/\beta}}{t^{n+1-1/\beta} (1 + \tabs{x})^{\beta - \alpha}} & \le \expr{-\frac{\D}{\D t}}^n \pr(\tau_x > t) \le \frac{c_2(\alpha, n) \tabs{x}^{\beta - 1} (1 + t)^{1/\alpha - 1/\beta}}{t^{n+1-1/\beta} (1 + \tabs{x})^{\beta - \alpha}}
}
for all $n \ge 0$, $t > 0$ and $x \in \R \setminus \{0\}$ such that $t / \max(\tabs{x}^\alpha, \tabs{x}^\beta) \ge c_3(\alpha, n)$. Furthermore, finite and positive limits
\formula{
 & \lim_{x \to 0} \expr{\tabs{x}^{1 - \beta} (-\tfrac{\D}{\D t})^n \pr(\tau_x > t)} , && \lim_{t \to \infty} \expr{t^{n + 1 - 1 / \alpha} (-\tfrac{\D}{\D t})^n \pr(\tau_x > t)}
}
exist for all $n \ge 0$, $t > 0$ and $x \in \R \setminus \{0\}$. As in the previous example, the restriction $\alpha > 1$ is required by the scaling-type condition~\eqref{eq:reg:1}. If $\alpha \le 1 < \beta$, then $0 < \pr(\tau_x < \infty) < 1$ and the estimates of $\pr(\tau_x < t)$ are unknown. When $\beta \le 1$, then $\pr(\tau_x < \infty) = 0$.
\end{example}

\begin{example}
\label{ex:tauxest:other}
Let $X$ be the pure-jump L\'evy process with L\'evy--Khintchine exponent $\Psi(\xi) = \xi^2 (\log(1 + \xi^2))^{-1} - 1$ (see~\cite{bib:m13}). Since $\Psi$ is regularly varying with index $2$ both at $0$ and at infinity, it can be checked that both large-time and small-time scaling limits:
\formula{
 & (k^{-1/2} X_{k t} : t \ge 0) && \text{as $k \to \infty$}, \\
 & ((2 k)^{-1/2} X_{k \log(1/k) t} : t \ge 0) && \text{as $k \to 0^+$,}
}
are standard Wiener processes (cf.~\cite{bib:gk68}). Let $\ph(t) = 1$ for $t \ge \tfrac{1}{e}$ and $\ph(t) = (t e^{-W_{-1}(-t)})^{1/2}$ when $0 < t < \tfrac{1}{e}$ (where $W_{-1}$ is the lower branch of the Lambert $W$ function). We have
\formula{
 \frac{c_1(n) \tabs{x} \log(2 + \tfrac{1}{\tabs{x}})}{t^{n+1/2} \ph(t)} & \le \expr{-\frac{\D}{\D t}}^n \pr(\tau_x > t) \le \frac{c_2(n) \tabs{x} \log(2 + \tfrac{1}{\tabs{x}})}{t^{n+1/2} \ph(t)}
}
for all $n \ge 0$, $t > 0$ and $x \in \R \setminus \{0\}$ such that $t / (\tabs{x}^2 \log(2 + \tfrac{1}{\tabs{x}})) \ge c_3(n)$. Furthermore, finite and positive limits
\formula{
 & \lim_{x \to 0} \frac{(-\tfrac{\D}{\D t})^n \pr(\tau_x > t)}{\tabs{x} \log(2 + \tfrac{1}{\tabs{x}})} \, , && \lim_{t \to \infty} \expr{t^{n + 1/2} (-\tfrac{\D}{\D t})^n \pr(\tau_x > t)}
}
exist for all $n \ge 0$, $t > 0$ and $x \in \R \setminus \{0\}$.
\end{example}

\begin{example}\label{ex:counter}
Let $X$ be the sum of a standard Wiener process and a compound Poisson process with L\'evy measure $c e^{-|x|} \D x$. Then $X$ is symmetric, has completely monotone jumps and $\Psi(\xi) = \tfrac{1}{2} \xi^2 + c \xi^2 / (1 + \xi^2)$. By a direct calculation,
\formula{
 \frac{\xi \Psi''(\xi)}{\Psi'(\xi)} & = 1 - \frac{8 c \xi^2}{((1 + \xi^2)^2 + 2 c) (1 + \xi^2) } \, .
}
The right-hand side decreases with $c \ge 0$, and for $c = 2$ we have
\formula{
 \inf \left\{ \frac{\xi \Psi''(\xi)}{\Psi'(\xi)} : \xi \in (0, \infty) \right\} & = \frac{\Psi''(1)}{\Psi'(1)} = 0 .
}
It follows that $X$ satisfies the scaling-type condition~\eqref{eq:reg:1} with $\alpha \in (1, 2]$ if and only if $c \in [0, 2)$. We remark that the restriction $c < 2$ is apparently the limitation of our method, there is no reason to believe that for $c \ge 2$ the conclusions of Theorems~\ref{th:tauxest:reg} and~\ref{th:tauxasymp:reg} no longer hold.
\end{example}

\begin{remark}\label{rem:levy}
The scaling-type condition~\eqref{eq:reg:1} with $\alpha \in (1, 2]$ is easily shown to be equivalent to concavity of $\Psi(\xi^{1 - \eps})$ for some $\eps \in (0, \tfrac{1}{2}]$ (with $\alpha - 1 = \tfrac{\eps}{1 - \eps}$). A sufficient condition for~\eqref{eq:reg:1} with $\alpha \in (1, 2]$ in terms of the L\'evy measure of $X$ is described below.

Let $X$ be a symmetric L\'evy process with completely monotone jumps, and \revisiontwo{denote the density function of the L\'evy measure $\nu$ of $X$ by the same symbol $\nu$.} 
Then
\formula{
 \Psi(\xi) & = a \xi^2 + 2 \int_0^\infty (1 - \cos(\xi z)) \revisiontwo{\nu}(z) \D z = a \xi^2 + 2 \int_0^\infty (1 - \cos s) \tfrac{1}{\xi} \revisiontwo{\nu}(\tfrac{s}{\xi}) \D s .
}
Assuming that $\tfrac{\D}{\D \xi} (\tfrac{1}{\xi} \revisiontwo{\nu}(\tfrac{s}{\xi})) \ge 0$ and $\tfrac{\D^2}{\D \xi^2} (\tfrac{1}{\xi} \revisiontwo{\nu}(\tfrac{s}{\xi})) \ge 0$, differentiation in $\xi$ under the integral sign is permitted. It follows that
\formula{
 \xi^2 \Psi''(\xi) - (\alpha - 1) \xi \Psi'(\xi) & = 2 a (2 - \alpha) \xi^2 \\
 & \hspace*{-5em} + 2 \int_0^\infty (1 - \cos s) \tfrac{1}{\xi} ((\tfrac{s}{\xi})^2 \revisiontwo{\nu}''(\tfrac{s}{\xi}) + (3 + \alpha) \tfrac{s}{\xi} \revisiontwo{\nu}'(\tfrac{s}{\xi}) + (1 + \alpha) \revisiontwo{\nu}(\tfrac{s}{\xi})) \D s .
}
The right-hand side is non-negative if $z^2 \revisiontwo{\nu}''(z) + (3 + \alpha) z \revisiontwo{\nu}'(z) + (1 + \alpha) \revisiontwo{\nu}(z) \ge 0$ for all $z > 0$, which is equivalent to $\tfrac{\D^2}{\D z^2} (z^{-1/\alpha} \revisiontwo{\nu}(z^{-1/\alpha})) \ge 0$. This condition alone implies that $\tfrac{\D^2}{\D \xi^2} (\tfrac{1}{\xi} \revisiontwo{\nu}(\tfrac{s}{\xi})) \ge 0$, and if $z^{-1/\alpha} \revisiontwo{\nu}(z^{-1/\alpha})$ is increasing, then also $\tfrac{\D}{\D \xi} (\tfrac{1}{\xi} \revisiontwo{\nu}(\tfrac{s}{\xi})) \ge 0$.

The above argument shows that if $\alpha \in (1, 2]$ and $z^{-1/\alpha} \revisiontwo{\nu}(z^{-1/\alpha})$ is convex and nondecreasing in $z > 0$, then~\eqref{eq:reg:1} holds.
\end{remark}


Since the proofs of main theorems are rather technical, below we outline the main idea and briefly discuss the structure of the article. Our starting point is the following generalised eigenfunction expansion, proved in~\cite{bib:k12}. Note that in the original statement the condition $\xi \Psi''(\xi) \le \Psi'(\xi)$ was erroneously given as $2 \xi \Psi''(\xi) \le \Psi'(\xi)$ (the proof, however, used the correct condition). In the statement, as well as in the remaining part of the article, by $\fourier f(\xi) = \int_{-\infty}^\infty f(s) e^{-i s \xi} \D s$ we denote the Fourier transform of an integrable function~$f$. Occasionally, the distributional Fourier transform is used: if $f$ is a Schwartz distribution, then $\fourier f$ is again a Schwartz distribution, defined by $\tscalar{\fourier f, \ph} = \tscalar{f, \fourier \ph}$ for all \revisiontwo{$\ph$ in the} Schwartz class
.

\begin{theorem}[{\cite[Theorem~1.1 and Remark~1.2]{bib:k12}}]
\label{th:taux}
Suppose that $X$ is a symmetric L{\'e}vy process. If $1 / \Psi$ is integrable at infinity and
\formula[eq:reg:0]{
 \Psi'(\xi) & > 0 , & \frac{\xi \Psi''(\xi)}{\Psi'(\xi)} & \le 1
}
for all $\xi > 0$ (cf.~\eqref{eq:reg:1}), then
\formula[eq:taux]{
 (-\tfrac{\D}{\D t})^n \pr(t < \tau_x < \infty) & = \frac{1}{\pi} \int_0^\infty \cos \thet_\lambda e^{-t \Psi(\lambda)} \Psi'(\lambda) (\Psi(\lambda))^{n-1} F_\lambda(x) \D\lambda
}
for all $n \ge 0$ and $t > 0$, and almost all $x \in \R$. Here $F_\lambda$ is a bounded, continuous function, defined by
\formula{
 F_\lambda(x) & = \sin(\lambda \tabs{x} + \thet_\lambda) - G_\lambda(x)
}
for all $x \in \R$, where
\formula[eq:theta-def]{
 \thet_\lambda & = \arctan\expr{\frac{1}{\pi} \, \int_0^\infty \expr{\frac{\Psi'(\lambda)}{\Psi(\xi) - \Psi(\lambda)} - \frac{2 \lambda}{\xi^2 - \lambda^2}} \D \xi}
}
and $G_\lambda$ is an $L^2(\R) \cap C_0(\R)$ function with (integrable) Fourier transform
\formula{
 \fourier G_\lambda(\xi) & = \cos \thet_\lambda \expr{\frac{\Psi'(\lambda)}{\Psi(\xi) - \Psi(\lambda)} - \frac{2 \lambda}{\xi^2 - \lambda^2}}
}
for all $\xi \in \R \setminus \{-\lambda, \lambda\}$. The distributional Fourier transform of $F_\lambda$ is given by
\formula{
 \scalar{\fourier F_\lambda, \ph} & = \cos \thet_\lambda \pvint_{-\infty}^\infty \frac{\Psi'(\lambda) \ph(\xi)}{\Psi(\lambda) - \Psi(\xi)} \, \D \xi + \pi \sin \thet_\lambda (\ph(\lambda) + \ph(-\lambda))
}
for $\ph$ in the Schwartz class (here $\pvint$ stands for the Cauchy principal value integral).
\end{theorem}

As it is explained right after formula~\eqref{eq:reg} below, symmetric L\'evy processes with completely monotone jumps automatically satisfy~\eqref{eq:reg:0}, so Theorem~\ref{th:taux} can be applied whenever $1 / \Psi$ is integrable at infinity. The latter condition holds, for example, if the scaling-type condition~\eqref{eq:reg:1} is satisfied with $\alpha \in (1, 2]$.

The main idea of the proof of Theorems~\ref{th:tauxest:reg} and~\ref{th:tauxasymp:reg} is taken from~\cite{bib:kmr12}, where a similar problem for first passage times was studied. The \emph{generalised eigenfunctions} $F_\lambda(x)$ are oscillatory due to the $\sin(\lambda \tabs{x} + \thet_\lambda)$ term, but $F_\lambda(x) > 0$ when $\lambda \tabs{x}$ is small enough, and two-sided estimates for $F_\lambda(x)$ can be given in this case. Thanks to the exponential term $e^{-t \Psi(\lambda)}$ in~\eqref{eq:taux}, the main contribution to the integral comes from \smash{$\lambda \in (0, \tfrac{c}{\tabs{x}})$}, provided that $t$ is large enough, or $\tabs{x}$ is small enough. This essentially gives Theorem~\ref{th:tauxest:reg}. The proof of Theorem~\ref{th:tauxasymp:reg} requires in addition an asymptotic expression for $F_\lambda(x)$ as $x \to 0$ or $\lambda \to 0$.

We collect some simple technical results in Section~\ref{sec:pre}, so that they do not distract attention of the reader at a later point. In Section~\ref{sec:theta} the properties of $\thet_\lambda$ are studied. In Lemma~\ref{lem:thetalambdaest:reg} it is proved that the scaling-type condition~\eqref{eq:reg:1} implies $\thet_\lambda \le \tfrac{\pi}{\alpha} - \tfrac{\pi}{2}$ for all $\lambda > 0$. The asymptotic behaviour of $\thet_\lambda$ as $\lambda \to 0^+$ or $\lambda \to \infty$ is given in Lemma~\ref{lem:thetalambda:reg}.

The estimates and asymptotic properties of $F_\lambda$ are given in Section~\ref{sec:flambda}. Lemma~\ref{lem:flambdaest} contains a rather general estimate, which is then simplified in Lemma~\ref{lem:flambdaest:reg} for processes satisfying the scaling-type condition~\eqref{eq:reg:1}. Asymptotic \revisiontwo{expansions} 
of $F_\lambda$ \revisiontwo{are} 
given in Lemmas~\ref{lem:flambdareg} and~\ref{lem:fzero}.

The final Section~\ref{sec:taux} contains proofs of main theorems, preceded by two propositions of more general nature and two technical lemmas. Proposition~\ref{prop:everywhere} extends~\eqref{eq:taux} to all $x \in \R \setminus \{0\}$. Lemmas~\ref{lem:jest:reg} and~\ref{lem:iest:reg} contain estimates of the main part (\smash{$\lambda < \tfrac{c}{\tabs{x}}$}) and the remainder part (\smash{$\lambda > \tfrac{c}{\tabs{x}}$}) of the integral in~\eqref{eq:taux}.

Instead of using the L\'evy--Khintchine exponent $\Psi$, it is convenient to work with $\psi(\xi) = \Psi(\sqrt{\xi})$. Recall that when $X$ has completely monotone jumps, then $\psi$ is a complete Bernstein function. In the remaining part of the article $\Psi$ is virtually dropped from the notation. For reader's convenience, we note that
\formula[eq:psipsi]{
 \Psi(\xi) & = \psi(\xi^2), & \frac{\xi \Psi'(\xi)}{\Psi(\xi)} & = 2 \, \frac{\xi^2 \psi'(\xi^2)}{\psi(\xi^2)} \, , & \frac{\xi \Psi''(\xi)}{\Psi'(\xi)} & = 1 + 2 \, \frac{\xi^2 \psi''(\xi^2)}{\psi'(\xi^2)} \, ,
}
so that the scaling-type condition~\eqref{eq:reg:1} translates to
\formula{
 \frac{-\xi \psi''(\xi)}{\psi'(\xi)} & \le \frac{2 - \alpha}{2} \, .
}
To facilitate extensions, all intermediate results are stated for rather general functions~$\psi$. For this reason, statements of the results often contain assumptions, such as differentiability or monotonicity of $\psi$, which are automatically satisfied when $\psi$ corresponds to a symmetric L\'evy process with completely monotone jumps (that is, $\psi$ is a complete Bernstein function). In particular, in this more general setting, a two-sided scaling-type condition
\formula[eq:reg]{
 \frac{2 - \beta}{2} \le \frac{-\xi \psi''(\xi)}{\psi'(\xi)} & \le \frac{2 - \alpha}{2}
}
is often imposed. When $\psi$ is a complete Bernstein function, the lower bound in~\eqref{eq:reg} always holds with $\beta = 2$ (see~\cite[Proposition~2.21]{bib:k11}).

It should be pointed out that although we follow closely the approach of~\cite{bib:kmr12}, there are essential differences between the present problem and the one considered therein. The overall form of the generalised eigenfunctions is similar (sine term plus completely monotone correction $G_\lambda$), but the expressions for $\thet_\lambda$ and $G_\lambda$ are different, and thus require different methods. For example, the estimates of $G_\lambda$ in~\cite{bib:kmr12} follow easily from the expression for the Laplace transform of $G_\lambda$. We were unable to follow the same approach and needed to use Fourier transform instead. Also the technical details of the arguments are different, so virtually no part of~\cite{bib:kmr12} can be re-used in our setting.

%
%

\section{Preliminaries}
\label{sec:pre}

\revisionone{Throughout the article, by $c$, $c_1$, \revisiontwo{$c_2$, etc.} 
we denote positive constants. Dependence on a parameter $\alpha$ is always indicated by writing $c(\alpha)$\revisiontwo{,} etc.}

Following~\cite{bib:kmr12}, for $\lambda > 0$ and a continuous function $\psi : \hl \to \hl$ such that $\psi(\xi) \ne \psi(\lambda^2)$ when $\xi \ne \lambda^2$, we define
\formula{
 \psi_\lambda(\xi) & = \revisionone{\frac{1 - \tfrac{\xi}{\lambda^2}}{1 - \tfrac{\psi(\xi)}{\psi(\lambda^2)}}}
}
for $\xi > 0$, $\xi \ne \lambda^2$. This definition is extended continuously at $\xi = \lambda^2$ by $\psi_\lambda(\lambda^2) = \psi(\lambda^2) / (\lambda^2 \psi'(\lambda^2))$ whenever $\psi$ is differentiable at $\lambda^2$ and $\psi'(\lambda^2) > 0$. In this case we say that $\psi_\lambda$ is well-defined. 

If for some $\lambda > 0$ the function $\psi_\lambda$ is well-defined and $\psi_\lambda(\xi) \ne \psi_\lambda(\lambda^2)$ for $\xi \ne \lambda^2$, then $(\psi_\lambda)_\lambda$ can be defined, and
\formula[eq:psilambdalambda]{
 \frac{1}{(\psi_\lambda)_\lambda(\xi^2)} & = \frac{\lambda^2 \psi'(\lambda^2)}{\psi(\xi^2) - \psi(\lambda^2)} - \frac{\lambda^2}{\xi^2 - \lambda^2}
}
for $\xi > 0$, $\xi \ne \lambda^2$. Note that if $\psi : \hl \to \hl$ is twice differentiable and $\psi'(\xi) > 0$, $\psi''(\xi) < 0$ for all $\xi > 0$, then $\psi_\lambda$ is \revisionone{strictly} increasing for every $\lambda > 0$, and hence $(\psi_\lambda)_\lambda$ is well-defined \revisionone{and positive}. Furthermore, if $\psi$ is a complete Bernstein function \revisionone{(equivalently, if $\Psi(\xi) = \psi(\xi^2)$ is the L\'evy--Khintchine exponent of a symmetric L\'evy process with completely monotone jumps)}, then also $\psi_\lambda$ and $(\psi_\lambda)_\lambda$ are complete Bernstein functions (see~\cite{bib:k12, bib:ssv10}).

Below we list some rather elementary results used in the proofs of main results.

\begin{lemma}
\label{l:lambdalambda}
If $\psi, \tilde{\psi} : \hl \to \hl$ are twice differentiable, $\psi'(\xi), \tilde{\psi}'(\xi) > 0$ and $\psi''(\xi), \tilde{\psi}''(\xi) \le 0$ for all $\xi > 0$, and furthermore
\formula[eq:lambdalambda:1]{
 \frac{-\psi''(\xi)}{\psi'(\xi)} & \le \frac{-\tilde{\psi}''(\xi)}{\tilde{\psi}'(\xi)}
}
for all $\xi > 0$, then
\formula[eq:lambdalambda:2]{
 (\psi_\lambda)_\lambda(\xi^2) & \ge (\tilde{\psi}_\lambda)_\lambda(\xi^2)
}
for all $\lambda, \xi > 0$.
\end{lemma}

\begin{proof}
Integration of~\eqref{eq:lambdalambda:1} in $\xi$ gives
\formula{
 \frac{\psi'(\zeta)}{\psi'(\xi_1)} & \ge \frac{\tilde{\psi}'(\zeta)}{\tilde{\psi}'(\xi_1)}
}
when $0 < \xi_1 < \zeta$. By another integration in $\zeta$,
\formula{
 \frac{\psi(\xi_2) - \psi(\xi_1)}{\psi'(\xi_1)} & \ge \frac{\tilde{\psi}(\xi_2) - \tilde{\psi}(\xi_1)}{\tilde{\psi}'(\xi_1)}
}
when $0 < \xi_1 < \xi_2$. Substituting $\xi_1 = \lambda^2$ and $\xi_2 = \xi^2$, one gets
\formula{
 \frac{\lambda^2 \psi'(\lambda^2)}{\psi(\xi^2) - \psi(\lambda^2)} - \frac{\lambda^2}{\xi^2 - \lambda^2} & \le \frac{\lambda^2 \tilde{\psi}'(\lambda^2)}{\tilde{\psi}(\xi^2) - \tilde{\psi}(\lambda^2)} - \frac{\lambda^2}{\xi^2 - \lambda^2} \, ,
}
that is, \eqref{eq:lambdalambda:2}, provided that $0 < \lambda < \xi$. A similar argument can be given when $0 < \xi < \lambda$. The case $\lambda = \xi > 0$ follows by continuity.
\end{proof}

\begin{lemma}
\label{l:ratioest}
If $\psi : \hl \to \hl$ is twice differentiable, $\psi'(\xi) > 0$ for all $\xi > 0$, and \revisionone{the scaling-type condition~\eqref{eq:reg} holds} for some $\alpha, \beta > 0$ and all $\xi > 0$, then
\formula[eq:ratioest:2]{
 \frac{\alpha}{2} & \le \frac{\xi \psi'(\xi)}{\psi(\xi) - \psi(0^+)} \le \frac{\beta}{2}
}
for all $\xi > 0$, and
\formula[eq:ratioest:3]{
 \expr{\frac{\xi_1}{\xi_2}}^{1 - \frac{\alpha}{2}} & \le \frac{\psi'(\xi_2)}{\psi'(\xi_1)} \le \expr{\frac{\xi_1}{\xi_2}}^{1 - \frac{\beta}{2}} , \qquad & \expr{\frac{\xi_2}{\xi_1}}^{\frac{\alpha}{2}} & \le \frac{\psi(\xi_2) - \psi(0^+)}{\psi(\xi_1) - \psi(0^+)} \le \expr{\frac{\xi_2}{\xi_1}}^{\frac{\beta}{2}}
}
whenever $0 < \xi_1 < \xi_2$.
\end{lemma}

\begin{proof}
By~\revisionone{\eqref{eq:reg}}, if $0 < \xi_1 < \xi_2$,
\formula{
 \log \expr{\frac{\xi_2}{\xi_1}}^{1 - \frac{\alpha}{2}} & = \int_{\xi_1}^{\xi_2} \frac{1 - \tfrac{\alpha}{2}}{\zeta} \, \D \zeta \ge \int_{\xi_1}^{\xi_2} \frac{-\psi''(\zeta)}{\psi'(\zeta)} \, \D \zeta = \log \frac{\psi'(\xi_1)}{\psi'(\xi_2)} \, ,
}
proving the lower bound in the first part of~\eqref{eq:ratioest:3}. Hence,
\formula{
 \frac{\xi_2}{\tfrac{\alpha}{2}} & = \int_0^{\xi_2} \expr{\frac{\xi_2}{\xi_1}}^{1 - \frac{\alpha}{2}} \D \xi_1 \ge \int_0^{\xi_2} \frac{\psi'(\xi_1)}{\psi'(\xi_2)} \, \D \xi_1 = \frac{\psi(\xi_2) - \psi(0^+)}{\psi'(\xi_2)} \, ,
}
which shows the lower bound in~\eqref{eq:ratioest:2}. Furthermore,
\formula{
 \log \expr{\frac{\xi_2}{\xi_1}}^{\frac{\alpha}{2}} & = \int_{\xi_1}^{\xi_2} \frac{\tfrac{\alpha}{2}}{\zeta} \D \zeta \le \int_{\xi_1}^{\xi_2} \frac{\psi'(\zeta)}{\psi(\zeta) - \psi(0^+)} \, \D \zeta = \log \frac{\psi(\xi_2) - \psi(0^+)}{\psi(\xi_1) - \psi(0^+)} \, ,
}
proving the other lower bound in~\eqref{eq:ratioest:3}. The upper bounds are proved in the same way.
\end{proof}

\revisionone{When $\psi(\xi^2)$ is the L\'evy--Khintchine exponent of a L\'evy process, then $\psi(0^+) = 0$. Hence, the latter part of~\eqref{eq:ratioest:3} takes \revisiontwo{the simpler} form}
\formula{
 \expr{\frac{\xi_2}{\xi_1}}^{\frac{\alpha}{2}} & \le \frac{\psi(\xi_2)}{\psi(\xi_1)} \le \expr{\frac{\xi_2}{\xi_1}}^{\frac{\beta}{2}} \revisionone{.}
}
\revisionone{Note that in this case}
\formula[eq:invest:2]{
 \expr{\frac{t_2}{t_1}}^{\frac{2}{\beta}} & \le \frac{\psi^{-1}(t_2)}{\psi^{-1}(t_1)} \le \expr{\frac{t_2}{t_1}}^{\frac{2}{\alpha}}
}
for all $t_1, t_2 > 0$ such that $t_1 < t_2$.

\begin{lemma}
\label{l:intlim}
If $g : \hl \to \hl$ is integrable and decreasing, then
\formula{
 \lim_{\xi \to \infty} (\xi g(\xi)) & = 0 .
}
\end{lemma}

\begin{proof}
\revisionone{As an integrable and decreasing function, $g(\xi)$ converges to $0$ as $\xi \to \infty$.} Since $g(\xi) \ind_{(0, \xi)}(\zeta) \le g(\zeta)$ for all $\xi, \zeta > 0$, by \revisiontwo{the Dominated Convergence Theorem}
,
\formula{
 \lim_{\xi \to \infty} (\xi g(\xi)) & = \lim_{\xi \to \infty} \int_0^\infty g(\xi) \ind_{(0, \xi)}(\zeta) \D \zeta = 0 . \qedhere
}
\end{proof}

\begin{lemma}
\label{l:gest}
If $g : \R \to \hl$ is integrable and decreasing on $\hl$, and $g(\xi) = g(-\xi)$ for $\xi > 0$, then
\formula[eq:gest]{
 \frac{1}{2} \int_0^\infty \min(\xi^2 x^2, 4) g(\xi) \D \xi & \le \fourier g(0) - \fourier g(x) \le \int_0^\infty \min(\xi^2 x^2, 4) g(\xi) \D \xi
}
for all $x \in \R$. Furthermore,
\formula[eq:gmodulus]{
 \tabs{\fourier g(x_1) - \fourier g(x_2)} & \le \int_0^\infty \min(\xi \tabs{x_1 - x_2}, 2) \min(\xi \tabs{x_1 + x_2}, 2) g(\xi) \D \xi
}
for all $x_1, x_2 \in \R$.
\end{lemma}

\begin{proof}
Fix $x > 0$. By symmetry of $g$,
\formula{
 \fourier g(0) - \fourier g(x) & = 2 \int_0^\infty (1 - \cos(\xi x)) g(\xi) \D \xi \revisiontwo{.}
}
Clearly, $1 - \cos(\xi x) \le 2$ and $1 - \cos(\xi x) = 2 \sin (\tfrac{\xi x}{2})^2 \le \tfrac{1}{2} \xi^2 x^2$. Therefore,
\formula{
 \fourier g(0) - \fourier g(x) & \le \int_0^{\frac{2}{x}} \xi^2 x^2 g(\xi) \D \xi + \int_{\frac{2}{x}}^\infty 4 g(\xi) \D \xi .
}
For the lower bound, integration by parts gives
\formula{
 \fourier g(0) - \fourier g(x) & = 2 \lim_{\xi \to \infty} \expr{(\xi - \tfrac{1}{x} \sin(\xi x)) g(\xi)} + 2 \int_0^\infty (\xi - \tfrac{1}{x} \sin(\xi x)) (-\D g(\xi)) ,
}
where the \revisiontwo{integral} 
in the right-hand side is a Lebesgue--Stieltjes one (if $g$ is differentiable, then $(-\D g(\xi)) = (-g'(\xi)) \D \xi$). By \revisionone{Lemma}~\ref{l:intlim}, the limit in the right-hand side is $0$. Furthermore, $(-\D g(\xi))$ is a non-negative measure on $(0, \infty)$, and one easily verifies that $\xi - \tfrac{1}{x} \sin(\xi x) \ge \tfrac{1}{8} \xi^3 x^2$ for $\xi \in (0, \tfrac{2}{x})$ and $\xi - \tfrac{1}{x} \sin(\xi x) \ge \xi - \tfrac{1}{x}$ for $\xi \in (\tfrac{2}{x}, \infty)$. Hence,
\formula{
 \fourier g(0) - \fourier g(x) & \ge \int_0^{\frac{2}{x}} \frac{\xi^3 x^2}{4} \, (-\D g(\xi)) + \int_{\frac{2}{x}}^\infty 2 (\xi - \tfrac{1}{x}) (-\D g(\xi)) .
}
The function $\tfrac{1}{4} \xi^3 x^2 \ind_{(0, 2/x)}(\xi) + 2 (\xi - \tfrac{1}{x}) \ind_{[2/x, \infty)}(\xi)$ is continuous at $\xi = \tfrac{2}{x}$. Therefore, another integration by parts gives
\formula{
 \fourier g(0) - \fourier g(x) & \ge \int_0^{\frac{2}{x}} \tfrac{3}{4} \xi^2 x^2 g(\xi) \D \xi + \int_{\frac{2}{x}}^\infty 2 g(\xi) \D \xi .
}
It follows that
\formula{
 \fourier g(0) - \fourier g(x) & \ge \frac{1}{2} \expr{\int_0^{\frac{2}{x}} \xi^2 x^2 g(\xi) \D \xi + \int_{\frac{2}{x}}^\infty 4 g(\xi) \D \xi} ,
}
as desired. The estimates~\eqref{eq:gest} for $x < 0$ follow by symmetry.

In a similar manner, for $x_1, x_2 \in \R$,
\formula{
 \tabs{\fourier g(x_1) - \fourier g(x_2)} & \le 2 \int_0^\infty \tabs{\cos(\xi x_1) - \cos(\xi x_2)} g(\xi) \D \xi \\
 & = 4 \int_0^\infty \tabs{\sin \tfrac{\xi x_1 - \xi x_2}{2}} \tabs{\sin \tfrac{\xi x_1 + \xi x_2}{2}} g(\xi) \D \xi ,
}
and~\eqref{eq:gmodulus} follows from $\tabs{\sin s} \le \min(s, 1)$ for $s > 0$.
\end{proof}

\begin{lemma}
\label{l:intest}
If $\psi : \hl \to \hl$ and $\xi / \psi(\xi)$ is increasing in $\xi > 0$, then
\formula[eq:intest]{
 \int_0^\xi \frac{\zeta^2}{\psi(\zeta^2)} \, \D \zeta & \le \xi^2 \int_\xi^\infty \frac{1}{\psi(\zeta^2)} \, \D \zeta
}
for all $\xi > 0$.
\end{lemma}

\begin{proof}
When $0 < \zeta < \xi$, then $\zeta^2 / \psi(\zeta^2) \le \xi^2 / \psi(\xi^2)$, and so
\formula{
 \int_0^\xi \frac{\zeta^2}{\psi(\zeta^2)} \, \D \zeta & \le \int_0^\xi \frac{\xi^2}{\psi(\xi^2)} \, \D \zeta = \frac{\xi^3}{\psi(\xi^2)} \, .
}
When $0 < \xi < \zeta$, then $\zeta / \psi(\zeta^2) \ge \xi^2 / \psi(\xi^2)$, so that
\formula{
 \int_\xi^\infty \frac{1}{\psi(\zeta^2)} \, \D \zeta & \ge \int_\xi^\infty \frac{\xi^2}{\zeta^2 \psi(\xi^2)} \, \D \zeta = \frac{\xi}{\psi(\xi^2)} \, .
}
Formula~\eqref{eq:intest} follows.
\end{proof}

\begin{lemma}
\label{l:greg}
If $g : \R \to \R$ is integrable and regularly varying at infinity with index $-\gamma$ for $\gamma \in (1, 3)$, and $g(x) = g(-x)$ for $x > 0$, then
\formula{
 \lim_{x \to 0^+} \expr{\frac{x}{g(1 / x)} \, (\fourier g(0) - \fourier g(x))} & = \frac{\pi}{\Gamma(\gamma) \tabs{\cos \tfrac{\gamma \pi}{2}}} \, .
}
\end{lemma}

\begin{proof}
Clearly, $\fourier g(x) = 2 \int_0^\infty g(\xi) \cos(\xi x) \D \xi$. By~\cite[Theorem~5]{bib:ss75},
\formula{
 \lim_{x \to 0^+} \expr{\frac{x}{g(1 / x)} \, (\fourier g(x) - \fourier g(0))} & = 2 \Gamma(1 - \gamma) \sin \tfrac{\gamma \pi}{2} ,
}
where for $\gamma = 2$ it is understood that the right-hand side is equal to $\pi$. Furthermore, $\Gamma(1 - \gamma) \Gamma(\gamma) = \pi / \sin(\gamma \pi)$.
\end{proof}

%
%

\section{Estimates of $\thet_\lambda$}
\label{sec:theta}

Recall that
\formula[eq:thetalambda]{
 \thet_\lambda & = \arctan\expr{\frac{1}{\pi} \int_0^\infty \frac{2}{\lambda} \, \frac{1}{(\psi_\lambda)_\lambda(\xi^2)} \, \D \xi}
}
for $\lambda > 0$.

\begin{lemma}
\label{lem:thetalambdaest:reg}
If $\psi(\xi^2)$ is the L\'evy--Khintchine exponent of a symmetric L\'evy process, $\psi'(\xi) > 0$ for all $\xi > 0$ and \revisionone{the scaling-type condition~\eqref{eq:reg} holds} for some $\alpha, \beta \in [1, 2]$ and all $\xi > 0$, then
\formula{
 \frac{\pi}{\beta} - \frac{\pi}{2} & \le \thet_\lambda \le \frac{\pi}{\alpha} - \frac{\pi}{2}
}
for all $\lambda > 0$.
\end{lemma}

\begin{proof}
If $\tilde{\psi}(\xi) = \xi^{\alpha/2}$, then $-\xi \tilde{\psi}''(\xi) / \tilde{\psi}'(\xi) = 1 - \tfrac{\alpha}{2}$. Hence, by \revisionone{Lemma}~\ref{l:lambdalambda},
\formula{
 (\psi_\lambda)_\lambda(\xi^2) & \ge (\tilde{\psi}_\lambda)_\lambda(\xi^2)
}
for all $\lambda, \xi > 0$. By~\eqref{eq:thetalambda}, it follows that $\thet_\lambda \le \tilde{\thet}_\lambda$, where $\tilde{\thet}_\lambda$ is defined as $\thet_\lambda$, but using $\tilde{\psi}$ instead of $\psi$. By~\cite[Example~5.1]{bib:kmr12}, $\tilde{\thet}_\lambda = \tfrac{\pi}{\alpha} - \tfrac{\pi}{2}$. This proves the upper bound. The lower one is obtained in a similar manner.
\end{proof}

\begin{lemma}
\label{lem:thetalambda:reg}
Suppose that $\psi(\xi^2)$ is the L\'evy--Khintchine exponent of a symmetric L\'evy process, $\psi'(\xi) > 0$ for all $\xi > 0$, and \revisionone{the scaling-type condition~\eqref{eq:reg} holds} for some $\alpha, \beta \in [1, 2]$ and all $\xi > 0$. If $\psi'$ is regularly varying at zero with index $\tfrac{\delta}{2} - 1$ for some $\delta \in [1, 2]$, then
\formula{
 \lim_{\lambda \to 0^+} \thet_\lambda & = \frac{\pi}{\delta} - \frac{\pi}{2} \, .
}
Similarly, if $\psi'$ is regularly varying at infinity with index $\tfrac{\gamma}{2} - 1$ for some $\gamma \in [1, 2]$, then
\formula{
 \lim_{\lambda \to \infty} \thet_\lambda & = \frac{\pi}{\gamma} - \frac{\pi}{2} \, .
}
\end{lemma}

\begin{proof}
Suppose that $\psi'$ is regularly varying at zero with index $\tfrac{\delta}{2} - 1$ and let $\tilde{\psi}(\xi) = \xi^{\alpha/2}$, so that $-\xi \tilde{\psi}''(\xi) / \tilde{\psi}'(\xi) = 1 - \tfrac{\alpha}{2}$. By Karamata's theorem ~\cite[Theorem~1.5.11]{bib:bgt87}, $\lim_{\lambda \to 0^+} (\lambda^2 \psi'(\lambda^2) / \psi(\lambda^2)) = \tfrac{\delta}{2}$ and $\psi$ is regularly varying at zero with index $\tfrac{\delta}{2}$.

By a substitution $\xi = \lambda s$,
\formula{
 \lim_{\lambda \to 0^+} \thet_\lambda & = \arctan \expr{\frac{1}{\pi} \lim_{\lambda \to 0^+} \int_0^\infty \revisionone{\frac{2}{(\psi_\lambda)_\lambda(\lambda^2 s^2)}} \, \D s} \\
 & = \arctan \expr{\frac{1}{\pi} \lim_{\lambda \to 0^+} \int_0^\infty \expr{\frac{2 \lambda^2 \psi'(\lambda^2) / \psi(\lambda^2)}{\psi(\lambda^2 s^2) / \psi(\lambda^2) - 1} - \frac{2}{s^2 - 1}} \D s} .
}
\revisiontwo{As $\lambda \to 0^+$, the} 
integrand converges pointwise to $\delta / (\revisionone{s^\delta - 1}) - 2 / (\revisionone{s^2 - 1})$\revisionone{. Furthermore, it is positive and bounded above by $2 / (\tilde{\psi}_\lambda)_\lambda(\lambda^2 s^2) = \alpha / (s^\alpha - 1) - 2 / (s^2 - 1)$ by Lemma~\ref{l:lambdalambda}. Note that this upper bound does not depend on $\lambda > 0$ and it is integrable in $s \in (0, \infty)$.} Hence, by \revisiontwo{the Dominated Convergence Theorem} 
and~\cite[Example~5.1]{bib:k12},
\formula{
 \lim_{\lambda \to 0^+} \thet_\lambda & = \arctan \expr{\frac{1}{\pi} \int_0^\infty \expr{\frac{\delta}{s^\delta - 1} - \frac{2}{s^2 - 1}} \D s} = \frac{\pi}{\delta} - \frac{\pi}{2} \, .
}
The other statement is proved in an analogous way.
\end{proof}

%
%

\section{Estimates of $F_\lambda(x)$}
\label{sec:flambda}

Recall that
\formula{
 \fourier G_\lambda(\xi) & = \frac{2 \cos \thet_\lambda}{\lambda} \, \frac{1}{(\psi_\lambda)_\lambda(\xi^2)} && \text{with} & \frac{1}{(\psi_\lambda)_\lambda(\xi^2)} & = \frac{\lambda^2 \psi'(\lambda^2)}{\psi(\xi^2) - \psi(\lambda^2)} - \frac{\lambda^2}{\xi^2 - \lambda^2}
}
for $\lambda > 0$, $\xi \in \R$, and
\formula{
 F_\lambda(x) & = \sin(\lambda \tabs{x} + \thet_\lambda) - G_\lambda(x)
}
for $\lambda > 0$, $x \in \R$.

\begin{lemma}
\label{lem:glambdaest}
If $\psi(\xi^2)$ is the L\'evy--Khintchine exponent of a symmetric L\'evy process, $1 / (1 + \psi(\xi^2))$ is integrable, $\lambda > 0$ and $(\psi_\lambda)_\lambda(\xi)$ is well-defined and increasing in $\xi > 0$, then
\formula{
 \frac{1}{4 \pi} \int_0^\infty \min(\xi^2 x^2, 4) \fourier G_\lambda(\xi) \D \xi & \le G_\lambda(0) - G_\lambda(x) \le \frac{1}{2 \pi} \int_0^\infty \min(\xi^2 x^2, 4) \fourier G_\lambda(\xi) \D \xi
}
for all $x \in \R$. Furthermore,
\formula{
 \tabs{G_\lambda(x_1) - G_\lambda(x_2)} & \le \frac{1}{2 \pi} \int_0^\infty \min(\xi \tabs{x_1 - x_2}, 2) \min(\xi \tabs{x_1 + x_2}, 2) \fourier G_\lambda(\xi) \D \xi
}
for all $x_1, x_2 \in \R$.
\end{lemma}

\begin{proof}
Due to symmetry of $G_\lambda$, $\fourier(\fourier G_\lambda) = 2 \pi G_\lambda$. Furthermore, $\fourier G_\lambda$ is differentiable and decreasing on $\hl$. Hence, the result follows by \revisionone{Lemma}~\ref{l:gest}.
\end{proof}

\begin{lemma}
\label{lem:glambdaest:reg}
If $\psi(\xi^2)$ is the L\'evy--Khintchine exponent of a symmetric L\'evy process, $1 / (1 + \psi(\xi^2))$ is integrable, $\lambda > 0$, $(\psi_\lambda)_\lambda(\xi)$ is well-defined and $(\psi_\lambda)_\lambda(\xi)$ and $\xi / (\psi_\lambda)_\lambda(\xi)$ are increasing in $\xi > 0$, then
\formula{
 \frac{1}{\pi} \int_{\frac{2}{\tabs{x}}}^\infty \fourier G_\lambda(\xi) \D \xi & \le G_\lambda(0) - G_\lambda(x) \le \frac{4}{\pi} \int_{\frac{2}{\tabs{x}}}^\infty \fourier G_\lambda(\xi) \D \xi
}
for all $x \in \R$.
\end{lemma}

\begin{proof}
Since $\xi / (\psi_\lambda)_\lambda(\xi)$ is increasing in $\xi > 0$, by \revisionone{Lemma}~\ref{l:intest},
\formula{
 \int_0^{\frac{2}{\tabs{x}}} \xi^2 x^2 \fourier G_\lambda(\xi) \D \xi & \le \int_{\frac{2}{\tabs{x}}}^\infty 4 \fourier G_\lambda(\xi) \D \xi .
}
The result follows now from Lemma~\ref{lem:glambdaest}.
\end{proof}

\begin{lemma}
\label{lem:flambdaest}
If $\psi(\xi^2)$ is the L\'evy--Khintchine exponent of a symmetric L\'evy process, $1 / (1 + \psi(\xi^2))$ is integrable, $\lambda > 0$, $(\psi_\lambda)_\lambda(\xi)$ is well-defined and $(\psi_\lambda)_\lambda(\xi)$ and $\xi / (\psi_\lambda)_\lambda(\xi)$ are increasing in $\xi > 0$, then
\formula{
 \frac{\cos \thet_\lambda}{\pi} \int_{\frac{2}{x}}^\infty \frac{2 \lambda \psi'(\lambda^2)}{\psi(\xi^2) - \psi(\lambda^2)} \, \D \xi & \le F_\lambda(x) \le \frac{4}{\pi} \int_{\frac{2}{x}}^\infty \frac{2 \lambda \psi'(\lambda^2)}{\psi(\xi^2) - \psi(\lambda^2)} \, \D \xi
}
for $\lambda, x > 0$ satisfying $\lambda x < \tfrac{\pi}{2} - \thet_\lambda$. The upper bound holds when $\lambda x < 2$.
\end{lemma}

\begin{proof}
Suppose that $\lambda, x > 0$ and write
\formula[eq:flambda:split]{
 F_\lambda(x) & = (\sin(\lambda x + \thet_\lambda) - \sin(\thet_\lambda)) + (G_\lambda(0) - G_\lambda(x)) .
}
By Lemma~\ref{lem:glambdaest:reg}, $G_\lambda(0) - G_\lambda(x)$ is bounded below and above by a constant times (see~\eqref{eq:psilambdalambda})
\formula{
 \int_{\frac{2}{x}}^\infty \fourier G_\lambda(\xi) \D \xi & = \cos \thet_\lambda \int_{\frac{2}{x}}^\infty \frac{2}{\lambda (\psi_\lambda)_\lambda(\xi^2)} \, \D \xi \\
 & = \cos \thet_\lambda \int_{\frac{2}{x}}^\infty \expr{\frac{2 \lambda \psi'(\lambda^2)}{\psi(\xi^2) - \psi(\lambda^2)} - \frac{2 \lambda}{\xi^2 - \lambda^2}} \D \xi
}
Observe that $\tfrac{\D}{\D s} (\log(1 + s) - \log(1 - s)) \ge \revisionone{2}$ for $s \in (0, 1)$. Therefore, if $\lambda x < 2$, then
\formula{
 \int_{\frac{2}{x}}^\infty \frac{2 \lambda}{\xi^2 - \lambda^2} \D \xi & = \log(1 + \tfrac{\lambda x}{2}) - \log(1 - \tfrac{\lambda x}{2}) \ge \lambda x \ge \sin(\lambda x + \thet_\lambda) - \sin(\thet_\lambda) .
}
Hence,
\formula{
 F_\lambda(x) & \le \int_{\frac{2}{x}}^\infty \frac{2 \lambda}{\xi^2 - \lambda^2} \D \xi + \frac{4 \cos \thet_\lambda}{\pi} \int_{\frac{2}{x}}^\infty \frac{2}{\lambda (\psi_\lambda)_\lambda(\xi^2)} \, \D \xi \\
 & \le \frac{4}{\pi} \int_{\frac{2}{x}}^\infty \frac{2 \lambda}{\xi^2 - \lambda^2} \D \xi + \frac{4}{\pi} \int_{\frac{2}{x}}^\infty \frac{2}{\lambda (\psi_\lambda)_\lambda(\xi^2)} \, \D \xi = \frac{4}{\pi} \int_{\frac{2}{x}}^\infty \frac{2 \lambda \psi'(\lambda^2)}{\psi(\xi^2) - \psi(\lambda^2)} \, \D \xi .
}
The lower bound is found in a similar manner. Observe that $\log(1 + s) - \log(1 - s)$ is convex on $(0, 1)$. Hence, if $\tfrac{\lambda x}{2} < \tfrac{\pi}{4}$, then
\formula{
 \int_{\frac{2}{x}}^\infty \frac{2 \lambda}{\xi^2 - \lambda^2} \D \xi & = \log(1 + \tfrac{\lambda x}{2}) - \log(1 - \tfrac{\lambda x}{2}) \\
 & \le (\log(1 + \tfrac{\pi}{4}) - \log (1 - \tfrac{\pi}{4})) \tfrac{4}{\pi} \, \tfrac{\lambda x}{2} = \tfrac{2}{\pi} (\log \tfrac{4 + \pi}{4 - \pi}) \lambda x .
}
Furthermore, by concavity, $\sin(s + \thet_\lambda) - \sin \thet_\lambda \ge s (1 - \sin \thet_\lambda) / (\tfrac{\pi}{2} - \thet_\lambda)$ for $s \in (0, \tfrac{\pi}{2} - \thet_\lambda)$. It follows that if $\lambda x < \tfrac{\pi}{2} - \thet_\lambda$, then
\formula{
 \int_{\frac{2}{x}}^\infty \frac{2 \lambda}{\xi^2 - \lambda^2} \D \xi & \le \tfrac{2}{\pi} (\log \tfrac{4 + \pi}{4 - \pi}) \, \frac{\tfrac{\pi}{2} - \thet_\lambda}{1 - \sin \thet_\lambda} \, (\sin(\lambda x + \thet_\lambda) - \sin(\thet_\lambda)) \\
 & \le \frac{4 \log \tfrac{4 + \pi}{4 - \pi}}{\pi \cos \thet_\lambda} \, (\sin(\lambda x + \thet_\lambda) - \sin(\thet_\lambda)) ;
}
the last inequality follows from the inequality $1 - \cos s \ge \tfrac{1}{2} s \sin s$ for $s \in (0, \tfrac{\pi}{2})$ (which is easily proved by differentiation) with $s = \tfrac{\pi}{2} - \thet_\lambda$. This gives the desired lower bound,
\formula{
 F_\lambda(x) & \ge \frac{\pi \cos \thet_\lambda}{4 \log \tfrac{4 + \pi}{4 - \pi}} \int_{\frac{2}{x}}^\infty \frac{2 \lambda}{\xi^2 - \lambda^2} \D \xi + \frac{\cos \thet_\lambda}{\pi} \int_{\frac{2}{x}}^\infty \frac{2}{\lambda (\psi_\lambda)_\lambda(\xi^2)} \, \D \xi \\
 & \ge \frac{\cos \thet_\lambda}{\pi} \int_{\frac{2}{x}}^\infty \frac{2 \lambda}{\xi^2 - \lambda^2} \D \xi + \frac{\cos \thet_\lambda}{\pi} \int_{\frac{2}{x}}^\infty \frac{2}{\lambda (\psi_\lambda)_\lambda(\xi^2)} \, \D \xi \\
 & = \frac{\cos \thet_\lambda}{\pi} \int_{\frac{2}{x}}^\infty \frac{2 \lambda \psi'(\lambda^2)}{\psi(\xi^2) - \psi(\lambda^2)} \, \D \xi .
\qedhere
}
\end{proof}

\begin{lemma}
\label{lem:flambdaest:reg}
If $\psi(\xi^2)$ is the L\'evy--Khintchine exponent of a symmetric L\'evy process, $\lambda > 0$, $(\psi_\lambda)_\lambda(\xi)$ and $\xi / (\psi_\lambda)_\lambda(\xi)$ are increasing in $\xi > 0$, and \revisionone{the scaling-type condition~\eqref{eq:reg} holds} for some $\alpha, \beta \in (1, 2]$ and all $\xi > 0$, then
\formula[eq:flambdaest:reg:2]{
 \frac{\alpha - 1}{\pi} \, \frac{\lambda \psi'(\lambda^2)}{x \psi(1/x^2)} & \le F_\lambda(x) \le \frac{40}{\pi (\alpha - 1)} \, \frac{\lambda \psi'(\lambda^2)}{x \psi(1/x^2)}
}
for $\lambda, x > 0$ satisfying $\lambda x < \pi - \tfrac{\pi}{\alpha}$. The upper bound holds when $\lambda x < 2$. Furthermore,
\formula[eq:flambdamodulus:reg]{
 \tabs{F_\lambda(x_1) - F_\lambda(x_2)} & \le 3 \lambda \tabs{x_1 - x_2} + \frac{2 \lambda \psi'(\lambda^2)}{\pi} \int_{2 \lambda}^\infty \frac{\min(\xi \tabs{x_1 - x_2}, 2) \min(\xi \tabs{x_1 + x_2}, 2)}{\psi(\xi^2)} \, \D \xi
}
for $\lambda > 0$ and $x_1, x_2 \in \R$.
\end{lemma}

Note that \revisionone{the scaling-type condition~\eqref{eq:reg}} implies that $1 / (1 + \psi(\xi^2))$ is integrable (by \revisionone{Lemma}~\ref{l:ratioest}) and that $(\psi_\lambda)_\lambda(\xi)$ is well-defined.

\begin{proof}
By Lemma~\ref{lem:thetalambdaest:reg}, $\thet_\lambda \le \tfrac{\pi}{\alpha} - \tfrac{\pi}{2}$. Hence, by Lemma~\ref{lem:flambdaest},
\formula{
 \frac{\sin \tfrac{\pi}{\alpha}}{\pi} \int_{\frac{2}{x}}^\infty \frac{2 \lambda \psi'(\lambda^2)}{\psi(\xi^2) - \psi(\lambda^2)} \, \D \xi & \le F_\lambda(x) \le \frac{4}{\pi} \int_{\frac{2}{x}}^\infty \frac{2 \lambda \psi'(\lambda^2)}{\psi(\xi^2) - \psi(\lambda^2)} \, \D \xi
}
for $\lambda, x > 0$ such that $\lambda x < \pi - \tfrac{\pi}{\alpha}$. In this case $\xi > \tfrac{2}{x}$ implies $\xi > 2 (\pi - \tfrac{\pi}{\alpha})^{-1} \lambda > \tfrac{4}{\pi} \lambda$, and hence, by \revisionone{Lemma}~\ref{l:ratioest}, $\psi(\lambda^2) \le (\tfrac{\pi}{4})^\alpha \psi(\xi^2)$. Therefore,
\formula{
 \int_{\frac{2}{x}}^\infty \frac{2 \lambda \psi'(\lambda^2)}{\psi(\xi^2)} \, \D \xi & \le \int_{\frac{2}{x}}^\infty \frac{2 \lambda \psi'(\lambda^2)}{\psi(\xi^2) - \psi(\lambda^2)} \, \D \xi \le \frac{1}{1 - (\tfrac{\pi}{4})^\alpha} \int_{\frac{2}{x}}^\infty \frac{2 \lambda \psi'(\lambda^2)}{\psi(\xi^2)} \, \D \xi .
}
Finally, again by \revisionone{Lemma}~\ref{l:ratioest},
\formula{
 \int_{\frac{2}{x}}^\infty \frac{1}{\psi(\xi^2)} \, \D \xi & \le \frac{1}{\psi(1/x^2)} \int_{\frac{2}{x}}^\infty \frac{1}{(\xi x)^\alpha} \, \D \xi = \frac{1}{(\alpha - 1) 2^{\alpha - 1} x \psi(1/x^2)} \, ,
}
and a similar lower bound is valid with $\alpha$ replaced by $\beta$. By combining the above estimates, one obtains
\formula{
 \frac{\sin \tfrac{\pi}{\alpha}}{\pi} \, \frac{2 \lambda \psi'(\lambda^2)}{(\beta - 1) 2^{\beta - 1} x \psi(1/x^2)} & \le F_\lambda(x) \le \frac{4}{\pi} \, \frac{1}{1 - (\tfrac{\pi}{4})^{\alpha}} \, \frac{2 \lambda \psi'(\lambda^2)}{(\alpha - 1) 2^{\alpha - 1} x \psi(1/x^2)} \, ,
}
and~\eqref{eq:flambdaest:reg:2} follows by elementary estimates: $\sin \tfrac{\pi}{\alpha} \ge (\alpha - 1)$, $(\beta - 1) 2^{\beta - 1} \le 2$, $2^{\alpha - 1} \ge 1$, $1 - (\tfrac{\pi}{4})^\alpha \ge 1 - \tfrac{\pi}{4} \ge \tfrac{1}{5}$.

Formula~\eqref{eq:flambdamodulus:reg} is proved in a similar way. By Lemma~\ref{lem:glambdaest} and~\eqref{eq:flambda:split}, for $\lambda > 0$ and $x_1, x_2 \in \R$,
\formula{
 \tabs{F_\lambda(x_1) - F_\lambda(x_2)} & \le \lambda x + \frac{1}{2 \pi} \int_0^\infty \frac{2}{\lambda (\psi_\lambda)_\lambda(\xi^2)} \, \min(\xi x, 2) \min(\xi y, 2) \D \xi
}
where for brevity $x = \tabs{x_1 - x_2}$ and $y = \tabs{x_1 + x_2}$. Since $(\psi_\lambda)_\lambda(\xi) \ge (\psi_\lambda)_\lambda(0) = 1$ for $\xi \in (0, 2 \lambda)$, and $1 / (\psi_\lambda)_\lambda(\xi) \le \lambda^2 / (\psi(\xi^2) - \psi(\lambda^2))$ for $\xi > 2 \lambda$,
\formula{
 \tabs{F_\lambda(x_1) - F_\lambda(x_2)} & \le \lambda x + \int_0^{2 \lambda} \frac{2 \xi x}{\pi \lambda} \, \D \xi + \frac{1}{\pi} \int_{2 \lambda}^\infty \frac{\lambda \psi'(\lambda^2)}{\psi(\xi^2) - \psi(\lambda^2)} \, \min(\xi x, 2) \min(\xi y, 2) \D \xi \\
 & \le 3 \lambda x + \frac{\lambda \psi'(\lambda^2)}{(1 - (\tfrac{1}{2})^\alpha) \pi} \int_{2 \lambda}^\infty \frac{\min(\xi x, 2) \min(\xi y, 2)}{\psi(\xi^2)} \, \D \xi ;
}
here the last inequality follows by \revisionone{Lemma}~\ref{l:ratioest}.
\end{proof}

\begin{lemma}
\label{lem:flambdareg}
If $\psi(\xi^2)$ is the L\'evy--Khintchine exponent of a symmetric L\'evy process, $\lambda > 0$, $(\psi_\lambda)_\lambda(\xi)$ is well-defined, and $\psi$ is regularly varying at infinity with index $\tfrac{\gamma}{2}$ for some $\gamma \in (1, 2]$, then
\formula{
 \lim_{x \to 0^+} (x \psi(1 / x^2) F_\lambda(x)) & = \frac{\lambda \psi'(\lambda^2) \cos \thet_\lambda}{\Gamma(\gamma) \tabs{\cos \tfrac{\gamma \pi}{2}}} \, .
}
\end{lemma}

Note that $1 / (1 + \psi(\xi^2))$ is integrable, because it is regularly varying at infinity with index $-\gamma$.

\begin{proof}
Recall that
\formula{
 \fourier G_\lambda(\xi) & = \frac{2 \cos \thet_\lambda}{\lambda} \, \frac{1}{(\psi_\lambda)_\lambda(\xi^2)} = \frac{2 \lambda \psi'(\lambda^2) \cos \thet_\lambda}{\psi(\xi^2) - \psi(\lambda^2)} - \frac{2 \lambda \cos \thet_\lambda}{\xi^2 - \lambda^2} \, ,
}
and that $G_\lambda = \tfrac{1}{2 \pi} \fourier(\fourier G_\lambda)$. Let $a = \lim_{\xi \to \infty} (\psi(\xi^2) / \xi^2)$; necessarily $a \in [0, \infty)$. Then
\formula{
 \lim_{\xi \to \infty} (\psi(\xi^2) \fourier G_\lambda(\xi)) & = 2 \lambda (\psi'(\lambda^2) - a) \cos \thet_\lambda .
}
Therefore, $\fourier G_\lambda(\xi)$ is regularly varying at infinity with index $-\gamma$, and by \revisionone{Lemma}~\ref{l:greg},
\formula{
 \lim_{x \to 0^+} (x \psi(1 / x^2) (G_\lambda(0) - G_\lambda(x)) & = 2 \lambda (\psi'(\lambda^2) - a) \cos \thet_\lambda \lim_{x \to 0^+} \frac{x (G_\lambda(0) - G_\lambda(x))}{\fourier G_\lambda(1/x)} \\
 & = \frac{\lambda (\psi'(\lambda^2) - a) \cos \thet_\lambda}{\Gamma(\gamma) \tabs{\cos \tfrac{\gamma \pi}{2}}} \, .
}
Furthermore,
\formula{
 \lim_{x \to 0^+} (x \psi(1 / x^2) (\sin(\lambda x + \thet_\lambda) - \sin \thet_\lambda)) & = \lambda a \cos \thet_\lambda ,
}
and $F_\lambda(x) = (\sin(\lambda x + \thet_\lambda) - \sin \thet_\lambda) + (G_\lambda(0) - G_\lambda(x))$. The result clearly follows when $a = 0$. If $a > 0$, then necessarily $\gamma = 2$, and hence $\Gamma(\gamma) \tabs{\cos \tfrac{\gamma \pi}{2}} = 1$.
\end{proof}

Recall that the compensated potential kernel $v$ of $\revisionone{X}$ is defined by
\formula{
 v(x) & = \int_0^\infty (p_t(0) - p_t(x)) \D t ,
}
where $p_t(x)$ is the density function of the distribution of $X_t$. Since $\fourier p_t(\xi) = e^{-t \psi(\xi^2)}$, the distributional Fourier transform of $v$ \revisionone{satisfies}
\formula{
 \scalar{\fourier v, \ph} & = \int_0^\infty \int_{-\infty}^\infty e^{-t \psi(\xi^2)}(\ph(0) - \ph(\xi)) \D \xi \D t \\
 & = \int_0^\infty \int_0^\infty e^{-t \psi(\xi^2)}(2 \ph(0) - \ph(\xi) - \ph(-\xi)) \D \xi \D t \\
 & = \int_0^\infty \frac{2 \ph(0) - \ph(\xi) - \ph(-\xi)}{\psi(\xi^2)} \, \D \xi
}
for $\ph$ in the Schwartz class (\revisiontwo{the} Fubini \revisiontwo{theorem} is used in the last equality).

\begin{lemma}
\label{lem:fzero}
If $\psi(\xi^2)$ is the L\'evy--Khintchine exponent of a symmetric L\'evy process, $\lambda > 0$, $(\psi_\lambda)_\lambda(\xi)$ and $\xi / (\psi_\lambda)_\lambda(\xi)$ are increasing in $\xi > 0$, and \revisionone{the scaling-type condition~\eqref{eq:reg} holds} for some $\alpha, \beta \in (1, 2]$ and all $\xi > 0$, then
\formula{
 \lim_{\lambda \to 0^+} \frac{F_\lambda(x)}{2 \lambda \psi'(\lambda^2) \cos \thet_\lambda} & = v(x)
}
locally uniformly in $x \in \R$, where $v(x)$ is the compensated potential kernel of $\revisionone{X}$.
\end{lemma}

Noteworthy, convergence in the space of tempered distributions holds in full generality, that is, with the hypotheses of Theorem~\ref{th:taux}. Under the assumptions of the lemma, one also has $\thet_\lambda \to \tfrac{\pi}{\gamma} - \tfrac{\pi}{2}$ as $\lambda \to 0^+$ by Lemma~\ref{lem:thetalambda:reg}. As before, \revisionone{the scaling-type condition~\eqref{eq:reg}} implies that $1 / (1 + \psi(\xi^2))$ is integrable (by \revisionone{Lemma}~\ref{l:ratioest}) and that $(\psi_\lambda)_\lambda(\xi)$ is well-defined.

\begin{proof}
By Theorem~\ref{th:taux}, for $\ph$ in the Schwartz class,
\formula{
 \scalar{\frac{\fourier F_\lambda}{\lambda \psi'(\lambda^2) \cos \thet_\lambda}, \ph} & = 2 \pvint_{-\infty}^\infty \frac{\ph(\xi)}{\psi(\lambda^2) - \psi(\xi^2)} \D \xi + \frac{\pi \tan \thet_\lambda}{\lambda \psi'(\lambda^2)} (\ph(\lambda) + \ph(-\lambda)) \\
 & = 2 \pvint_{-\infty}^\infty \frac{\ph(\xi)}{\psi(\lambda^2) - \psi(\xi^2)} \D \xi - 2 \pvint_0^\infty \frac{\ph(\lambda) + \ph(-\lambda)}{\psi(\lambda^2) - \psi(\xi^2)} \, \D \xi \\
 & = 2 \int_0^\infty \frac{\ph(\xi) - \ph(\lambda) + \ph(-\xi) - \ph(-\lambda)}{\psi(\lambda^2) - \psi(\xi^2)} \, \D \xi .
}
As $\lambda \to 0^+$, the integrand converges pointwise to $(2 \ph(0) - \ph(\xi) - \ph(-\xi)) / \psi(\xi^2)$. We claim that \revisiontwo{the Dominated Convergence Theorem} 
applies to the above limit. Indeed,
\formula{
 \tabs{\ph(\xi) - \ph(\lambda) + \ph(-\xi) - \ph(-\lambda)} & \le \tabs{\xi - \lambda} \sup \{ \tabs{\ph'(s) - \ph'(-s)} : 0 < s < \xi + \lambda \} \\
 & \le \tabs{\xi - \lambda} (\xi + \lambda) \tnorm{\ph''}_\infty = \tabs{\xi^2 - \lambda^2} \tnorm{\ph''}_\infty ,
}
for all $\lambda, \xi > 0$, and since $\psi'$ is decreasing,
\formula{
 \tabs{\psi(\lambda^2) - \psi(\xi^2)} & \ge \tabs{\xi^2 - \lambda^2} \psi'(2)
}
for all $\lambda \in (0, 1)$ and $\xi \in (0, 2)$. Hence,
\formula{
 \abs{\frac{\ph(\xi) - \ph(\lambda) + \ph(-\xi) - \ph(-\lambda)}{\psi(\lambda^2) - \psi(\xi^2)}} & \le \frac{\tnorm{\ph''}_\infty}{\psi'(2)}
}
for all $\lambda \in (0, 1)$ and $\xi \in (0, 2)$. On the other hand,
\formula{
 \abs{\frac{\ph(\xi) - \ph(\lambda) + \ph(-\xi) - \ph(-\lambda)}{\psi(\lambda^2) - \psi(\xi^2)}} & \le \frac{4 \tnorm{\ph}_\infty}{\psi(\xi^2) - \psi(1)}
}
for all $\lambda \in (0, 1)$ and $\xi \ge 2$. The upper bound found above is integrable in $\xi \in (0, \infty)$, and the claim is proved. It follows that
\formula{
 \lim_{\lambda \to 0^+} \scalar{\frac{\fourier F_\lambda}{\lambda \psi'(\lambda^2) \cos \thet_\lambda}, \ph} & = 2 \scalar{\fourier v, \ph}
}
for every $\ph$ in the Schwartz class. This proves the desired result, but with locally uniform convergence replaced by convergence in the space of tempered distributions.

By Lemmas~\ref{lem:flambdaest:reg} and~\ref{lem:thetalambdaest:reg}, for all $\lambda > 0$ and $x_1, x_2 \in \R$,
\formula{
 \frac{\tabs{F_\lambda(x_1) - F_\lambda(x_2)}}{\lambda \psi'(\lambda^2) \cos \thet_\lambda} & \le \frac{3 \tabs{x_1 - x_2}}{\psi'(\lambda) \sin \tfrac{\pi}{\alpha}} + \frac{2}{\pi \sin \tfrac{\pi}{\alpha}} \int_{2 \lambda}^\infty \frac{\min(\xi \tabs{x_1 - x_2}, 2) \min(\xi \tabs{x_1 + x_2}, 2)}{\psi(\xi^2)} \, \D \xi .
}
Hence, if $\lambda \in (0, \lambda_0)$ and $x_1, x_2 \in [-x_0, x_0]$, then
\formula{
 \frac{\tabs{F_\lambda(x_1) - F_\lambda(x_2)}}{\lambda \psi'(\lambda^2) \cos \thet_\lambda} & \le \frac{3 \tabs{x_1 - x_2}}{\psi'(\lambda_0) \sin \tfrac{\pi}{\alpha}} + \frac{2}{\pi \sin \tfrac{\pi}{\alpha}} \int_0^\infty \frac{\min(\xi \tabs{x_1 - x_2}, 2) \min(2 \xi x_0, 2)}{\psi(\xi^2)} \, \D \xi .
}
The right-hand side is finite and converges to $0$ as $\tabs{x_2 - x_1} \to 0^+$ by \revisiontwo{the Dominated Convergence Theorem}
. Hence, the functions $F_\lambda(x) / (\lambda \psi'(\lambda^2) \cos \thet_\lambda)$ are equicontinuous in $x \in [-x_0, x_0]$ for $\lambda \in (0, \lambda_0)$. It remains to note that on a bounded interval, distributional convergence and equicontinuity imply uniform convergence.
\end{proof}

%
%

\section{Estimates of hitting times}
\label{sec:taux}

We begin with two technical results.

\begin{proposition}
\label{prop:tauxcont}
If $\revisionone{X}$ is a symmetric L\'evy process with L\'evy--Khintchine exponent $\Psi$, and $1 / (1 + \Psi(\xi))$ is integrable, then $\pr(t < \tau_x < \infty)$ is jointly continuous in $t > 0$ and $x \in \R$.
\end{proposition}

\begin{proof}
By~\cite[Theorem~43.5 and Remark~43.6]{bib:s99}, $\ex e^{-\lambda \tau_x}$ is a continuous function of $x \in \R$ for every $\lambda > 0$. Therefore, the distributions of $\tau_x$ are continuous in $x$ with respect to vague convergence of measures. It follows that the function $\pr(t < \tau_x < \infty)$ is continuous in $x$ at every point $(t, x)$ at which it is continuous in $t$.

Since $\pr(\tau_x = t) \le \pr(X_t = x) = 0$, the function $\pr(t < \tau_x < \infty)$ is continuous and non-increasing in $t > 0$ for every $x \in \R$. This implies that it is in fact jointly continuous in $t > 0$ and $x \in \R$.
\end{proof}

\begin{proposition}
\label{prop:everywhere}
If $\psi(\xi^2)$ is the L\'evy--Khintchine exponent of a symmetric L\'evy process, $1 / (1 + \psi(\xi^2))$ is integrable, $(\psi_\lambda)_\lambda$ is well-defined and $(\psi_\lambda)_\lambda(\xi)$ and $\xi / (\psi_\lambda)_\lambda(\xi)$ are increasing in $\xi > 0$ for all $\lambda > 0$, then equation~\eqref{eq:taux} in Theorem~\ref{th:taux} holds for all $x \in \R$ (and not just for almost all $x \in \R$).
\end{proposition}

\begin{proof}
It suffices to consider $n = 0$, the result for $n > 0$ follows then by differentiation, see~\cite[Remark~1.2]{bib:k12}. Let $t > 0$. By Proposition~\ref{prop:tauxcont}, the left-hand side of~\eqref{eq:taux} is a continuous function of $x \in \R \setminus \{0\}$. For each $t > 0$, the integrand in the right-hand side of~\eqref{eq:taux} is continuous in $x \in \R \setminus \{0\}$. Therefore, it remains to show that \revisiontwo{the Dominated Convergence Theorem} 
can be applied to prove continuity of the right-hand side of~\eqref{eq:taux} in $x > 0$ (equality for $x = 0$ is trivial, and the result for $x < 0$ follows by symmetry).

Fix $[a, b] \sub (0, \infty)$. By Lemma~\ref{lem:flambdaest}, for $x \in [a, b]$ and $\lambda \in (0, \tfrac{1}{b})$,
\formula{
 \abs{\cos \thet_\lambda e^{-t \psi(\lambda^2)} 2 \lambda \psi'(\lambda^2) (\psi(\lambda^2))^{-1} F_\lambda(x)} & \le \frac{4}{\pi} \, \frac{(2 \lambda \psi'(\lambda^2))^2}{\psi(\lambda^2)} \int_{\frac{2}{b}}^\infty \frac{1}{\psi(\xi^2) - \psi(1 / b^2)} \, \D \xi ,
}
while for $x \in [a, b]$ and $\lambda \ge \tfrac{1}{b}$,
\formula{
 \abs{\cos \thet_\lambda e^{-t \psi(\lambda^2)} 2 \lambda \psi'(\lambda^2) (\psi(\lambda^2))^{-1} F_\lambda(x)} & \le \frac{4 \lambda \psi'(\lambda^2)}{\psi(\lambda^2)} \, e^{-t \psi(\lambda^2)} ,
}
because $\tabs{F_\lambda(x)} \le 2$ (indeed, $\tabs{F_\lambda(x)} \le 1 + \tabs{G_\lambda(x)}$; since $\fourier G_\lambda(\xi) \ge 0$ for all $\xi \in \R$, one has $\tabs{G_\lambda(x)} \le G_\lambda(0)$; finally, $G_\lambda(0) = \sin \thet_\lambda \le 1$; see~\cite[Theorem~1.9(a)]{bib:k12}). Clearly,
\formula{
 \int_{\frac{2}{b}}^\infty \frac{4 \lambda \psi'(\lambda^2)}{\psi(\lambda^2)} \, e^{-t \psi(\lambda^2)} \D \lambda & = \int_{\psi(4 / b^2)}^\infty \frac{2 e^{-t s}}{s} \, \D s < \infty .
}
Furthermore, $\lambda^2 \psi'(\lambda^2) / \psi(\lambda^2) = 1 / \psi_\lambda(\lambda^2) \le 1 / \psi_\lambda(0) = 1$, and therefore
\formula{
 & \int_0^{\frac{2}{b}} \expr{\frac{4}{\pi} \, \frac{(2 \lambda \psi'(\lambda^2))^2}{\psi(\lambda^2)} \int_{\frac{2}{b}}^\infty \frac{1}{\psi(\xi^2) - \psi(1 / b^2)} \, \D \xi} \D \lambda \\
 & \hspace*{8em} \le \frac{16}{\pi} \int_0^{\frac{2}{b}} \psi'(\lambda^2) \D \lambda \int_{\frac{2}{b}}^\infty \frac{1}{\psi(\xi^2) - \psi(1 / b^2)} \, \D \xi < \infty ,
}
which completes the proof.
\end{proof}

By Proposition~\ref{prop:everywhere}, under appropriate assumptions, for $n \ge 0$, $t > 0$ and $x \in \R \setminus \{0\}$,
\formula{
 \expr{-\frac{\D}{\D t}}^n \pr(t < \tau_x < \infty) & = \frac{2}{\pi} \int_0^\infty \cos \thet_\lambda e^{-t \psi(\lambda^2)} \lambda \psi'(\lambda^2) (\psi(\lambda^2))^{n-1} F_\lambda(x) \D \lambda .
}
Throughout this section we denote
\formula[eq:ij]{
 I_n(t, x, a) & = \frac{2}{\pi} \int_a^\infty \cos \thet_\lambda e^{-t \psi(\lambda^2)} \lambda \psi'(\lambda^2) (\psi(\lambda^2))^{n-1} F_\lambda(x) \D \lambda , \\
 J_n(t, x, a) & = \frac{2}{\pi} \int_0^a \cos \thet_\lambda e^{-t \psi(\lambda^2)} \lambda \psi'(\lambda^2) (\psi(\lambda^2))^{n-1} F_\lambda(x) \D \lambda .
}

\revisionone{In the remaining part of the article, $\gammal(k; z)$ and $\gammau(k; z)$ denote the lower and the upper incomplete gamma functions, respectively.}

\begin{lemma}
\label{lem:iest:reg}
If $\psi(\xi^2)$ is the L\'evy--Khintchine exponent of a symmetric L\'evy process, $1 / (1 + \psi(\xi^2))$ is integrable, $(\psi_\lambda)_\lambda(\xi)$ is well-defined and $(\psi_\lambda)_\lambda(\xi)$ and $\xi / (\psi_\lambda)_\lambda(\xi)$ are increasing in $\xi > 0$ for all $\lambda > 0$, and \revisionone{the scaling-type condition~\eqref{eq:reg} holds} for some $\alpha, \beta \in (1, 2]$ and all $\xi > 0$, then
\formula{
 \tabs{I_n(t, x, a)} & \le \tfrac{2}{\pi} t^{-n} \gammau(n; (\alpha - 1)^\beta t \psi(1 / x^2))
}
for all $n \ge 0$, $t, x > 0$ and $a \ge (\pi - \tfrac{\pi}{\alpha}) / x$.
\end{lemma}

\begin{proof}
Fix $t, x > 0$ and let $a_0 = (\pi - \tfrac{\pi}{\alpha}) / x$ and $b_0 = t \psi(a_0^2)$. Using $\tabs{F_\lambda(x)} \le 2$ (see the proof of Proposition~\ref{prop:everywhere}) and a substitution $s = t \psi(\lambda^2)$, one finds that
\formula{
 \tabs{I_n(t, x, a)} & \le \frac{4}{\pi} \int_{a_0}^\infty e^{-t \psi(\lambda^2)} \lambda \psi'(\lambda^2) (\psi(\lambda^2))^{n-1} \D \lambda = \frac{2}{\pi t^n} \int_{b_0}^\infty e^{-s} s^{n-1} \D s = \frac{2 \gammau(n; b_0)}{\pi t^n} \, .
}
Furthermore, $b_0 = t \psi(a_0^2) \ge t \psi(1/x^2)$ if $\alpha > \tfrac{\pi}{\pi - 1}$, and $b_0 = t \psi(a_0^2) \ge (\pi - \tfrac{\pi}{\alpha})^\beta t \psi(1/x^2)$ otherwise (by \revisionone{Lemma}~\ref{l:ratioest}). In either case, $b_0 \ge (\alpha - 1)^\beta t \psi(1 / x^2)$.
\end{proof}

\begin{lemma}
\label{lem:jest:reg}
If $\psi(\xi^2)$ is the L\'evy--Khintchine exponent of a symmetric L\'evy process, $(\psi_\lambda)_\lambda(\xi)$ and $\xi / (\psi_\lambda)_\lambda(\xi)$ are increasing in $\xi > 0$ for all $\lambda > 0$, and \revisionone{the scaling-type condition~\eqref{eq:reg} holds} for some $\alpha, \beta \in (1, 2]$ and all $\xi > 0$, then there are constants $c_1(\alpha, \beta, n), c_2(\alpha, \beta, n) > 0$ such that
\formula{
 \frac{c_1(\alpha, \beta, n)}{t^{n+1} x \psi(1/x^2) \sqrt{\psi^{-1}(1/t)}} & \le J_n(t, x, (\pi - \tfrac{\pi}{\alpha}) / x) \le \frac{c_2(\alpha, \beta, n)}{t^{n+1} x \psi(1/x^2) \sqrt{\psi^{-1}(1/t)}}
}
for $n \ge 0$ and $t, x > 0$ such that $t \psi(1/x^2) \ge 1$. Here
\formula[eq:jest:reg:3]{
 c_1(\alpha, \beta, n) & = \frac{(\alpha - 1)^2 \gammal(n + 1 - \tfrac{1}{\beta}; (\alpha - 1)^\beta)}{2 \pi^2} \, , \\
 c_2(\alpha, \beta, n) & = \frac{40 (\gammal(n + 1 - \tfrac{1}{\alpha}; 1) + \gammau(n + 1 - \tfrac{1}{\beta}; 1))}{\pi^2 (\alpha - 1)} \, .
}
\end{lemma}

As before, \revisionone{the scaling-type condition~\eqref{eq:reg}} implies that $1 / (1 + \psi(\xi^2))$ is integrable (by \revisionone{Lemma}~\ref{l:ratioest}) and that $(\psi_\lambda)_\lambda(\xi)$ is well-defined.

\begin{proof}
Fix $t, x > 0$ and let $a = (\pi - \tfrac{\pi}{\alpha}) / x$ and $b = t \psi(a^2)$. Denote $J = J_n(t, x, a)$. Observe that when $\lambda < a$, then $\lambda x < \pi - \tfrac{\pi}{\alpha}$ and Lemma~\ref{lem:flambdaest:reg} applies. Hence,
\formula{
 J & \ge \frac{2 (\alpha - 1)}{\pi^2} \int_0^a \cos \thet_\lambda e^{-t \psi(\lambda^2)} \lambda \psi'(\lambda^2) (\psi(\lambda^2))^{n-1} \, \frac{\lambda \psi'(\lambda^2)}{x \psi(1/x^2)} \, \D \lambda .
}
Using $\cos \thet_\lambda \ge \cos(\tfrac{\pi}{\alpha} - \tfrac{\pi}{2}) \ge (\alpha - 1)$, $\lambda^2 \psi'(\lambda^2) \ge \tfrac{\alpha}{2} \psi(\lambda^2)$ (by \revisionone{Lemma}~\ref{l:ratioest}) and a substitution $s = t \psi(\lambda^2)$,
\formula{
 J & \ge \frac{\alpha (\alpha - 1)^2}{\pi^2 x \psi(1/x^2)} \int_0^a e^{-t \psi(\lambda^2)} \psi'(\lambda^2) (\psi(\lambda^2))^n \D \lambda = \frac{\alpha (\alpha - 1)^2}{\pi^2 t^{n+1} x \psi(1/x^2)} \int_0^b \frac{e^{-s} s^n}{2 \sqrt{\psi^{-1}(s/t)}} \, \D s .
}
By \revisionone{Lemma}~\ref{l:ratioest} and~\revisionone{\eqref{eq:invest:2}},
\formula{
 J & \ge \frac{\alpha (\alpha - 1)^2}{2 \pi^2 t^{n+1} x \psi(1/x^2)} \int_0^b \frac{e^{-s} s^n}{\sqrt{\max(s^{2/\beta}, s^{2/\alpha}) \psi^{-1}(1/t)}} \, \D s \\
 & \ge \frac{\alpha (\alpha - 1)^2 \gammal(n + 1 - \tfrac{1}{\beta}; \min(b, 1))}{2 \pi^2 t^{n+1} x \psi(1/x^2) \sqrt{\psi^{-1}(1/t)}} \, .
}
Finally, as in the proof of Lemma~\ref{lem:iest:reg}, $b \ge (\alpha - 1)^\beta t \psi(1/x^2)$. This proves the desired lower bound. The upper bound is shown in a similar manner,
\formula{
 J & \le \frac{80}{\pi^2 (\alpha - 1)} \int_0^a \cos \thet_\lambda e^{-t \psi(\lambda^2)} \lambda \psi'(\lambda^2) (\psi(\lambda^2))^{n-1} \, \frac{\lambda \psi'(\lambda^2)}{x \psi(1/x^2)} \, \D \lambda \\
 & \le \frac{40 \beta}{\pi^2 (\alpha - 1) x \psi(1/x^2)} \int_0^a e^{-t \psi(\lambda^2)} \psi'(\lambda^2) (\psi(\lambda^2))^n \D \lambda \\
 & = \frac{40 \beta}{\pi^2 (\alpha - 1) t^{n+1} x \psi(1/x^2)} \int_0^b \frac{e^{-s} s^n}{2 \sqrt{\psi^{-1}(s/t)}} \, \D s \\
 & \le \frac{20 \beta}{\pi^2 (\alpha - 1) t^{n+1} x \psi(1/x^2)} \int_0^b \frac{e^{-s} s^n}{\sqrt{\min(s^{2/\beta}, s^{2/\alpha}) \psi^{-1}(1/t)}} \, \D s \\
 & \le \frac{20 \beta (\gammal(n + 1 - \tfrac{1}{\alpha}; 1) + \gammau(n + 1 - \tfrac{1}{\beta}; 1))}{\pi^2 (\alpha - 1) t^{n+1} x \psi(1/x^2) \sqrt{\psi^{-1}(1/t)}} \, . \qedhere
}
\end{proof}

As observed in \revisionone{the introduction}, with the hypotheses of Theorems~\ref{th:tauxest:reg} and~\ref{th:tauxasymp:reg}, $\psi(\xi) = \Psi(\sqrt{\xi})$ is a complete Bernstein function (see~\cite{bib:ssv10}), and hence $\psi_\lambda$ and $(\psi_\lambda)_\lambda$ are complete Bernstein functions (see~\cite{bib:k12}). In particular, $(\psi_\lambda)_\lambda$ is well-defined and $(\psi_\lambda)_\lambda(\xi)$ and $\xi / (\psi_\lambda)_\lambda(\xi)$ are increasing in $\xi > 0$ for all $\lambda > 0$. Furthermore, if $\psi$ is regularly varying at zero or at infinity with index $\tfrac{\alpha}{2}$, then $\psi'$ is regularly varying at the same point with index $\tfrac{\alpha}{2} - 1$. \revisionone{The scaling-type condition~\eqref{eq:reg:1} implies the upper bound of~\eqref{eq:reg}, and the lower bound is automatically satisfied with $\beta = 2$.} Finally, $\Psi^{-1}(t) = (\psi^{-1}(t))^{1/2}$, and the relation between the derivatives of $\psi$ and $\Psi$ is given in~\eqref{eq:psipsi}.

Observe that the distributions of $\tau_x$ and $\tau_{-x}$ are equal, and $F_\lambda$ are even functions. Hence, only $x > 0$ needs to be considered in the proofs of main theorems.

\begin{proof}[Proof of Theorem~\ref{th:tauxest:reg}]
\revisionone{Let $\beta = 2$.} Choose $c > 1$ large enough, so that for $s \ge c$,
\formula{
 \tfrac{2}{\pi} s \, \gammau(n; (\alpha - 1)^\beta s) & \le \tfrac{1}{2} c_1(\alpha, \beta, n) ,
}
where $c_1(\alpha, \beta, n)$ is defined in~\eqref{eq:jest:reg:3} in Lemma~\ref{lem:jest:reg} (this is possible, because $\gammau(n; (\alpha - 1)^\beta s)$ decays exponentially fast with $s$ at infinity). Fix $t, x > 0$ such that $t \psi(1 / x^2) \ge c$, and let $a = (\pi - \tfrac{\pi}{\alpha}) / x$. Observe that
\formula{
 x^2 \psi^{-1}(\tfrac{1}{t}) & = \frac{\psi^{-1}(1 / t)}{\psi^{-1}(\psi(1 / x^2))} \le \frac{\psi^{-1}(1 / t)}{\psi^{-1}(c / t)} \le 1 .
}
Hence, by Lemmas~\ref{lem:iest:reg} and~\ref{lem:jest:reg}, if $t \psi(1 / x^2) \ge c$, then
\formula{
 t^n \tabs{I_n(t, x, a)} & \le \tfrac{2}{\pi} \gammau(n; (\alpha - 1)^\beta t \psi(1 / x^2)) \le \frac{c_1(\alpha, \beta, n)}{2 t \psi(1 / x^2)} \, , \\
 t^n J_n(t, x, a) & \ge \frac{c_1(\alpha, \beta, n)}{t x \psi(1 / x^2) \sqrt{\psi^{-1}(1/t)}} \ge \frac{c_1(\alpha, \beta, n)}{t \psi(1 / x^2)} \, ,
}
so that $\tabs{I_n(t, x, a)} \le \tfrac{1}{2} J_n(t, x, a)$. It follows that
\formula{
 \tfrac{1}{2} J_n(t, x, a) & \le \expr{-\frac{\D}{\D t}}^n \pr(t < \tau_x < \infty) \le \tfrac{3}{2} J_n(t, x, a) ,
}
and the theorem follows now directly from Lemma~\ref{lem:jest:reg}, with $C_1(\alpha, n) = \tfrac{1}{2} c_1(\alpha, \beta, n)$, $C_2(\alpha, n) = \tfrac{3}{2} c_2(\alpha, \beta, n)$ (see~\eqref{eq:jest:reg:3} in Lemma~\ref{lem:jest:reg}) and $C_3(\alpha, n) = c$.
\end{proof}

\begin{proof}[Proof of Corollary~\ref{cor:nzero}]
For brevity, denote the constants of Theorem~\ref{th:tauxest:reg} by $C_j = C_j(\alpha, 0)$ for $j = 1, 2, 3$; recall that $C_3 \ge 1$. Suppose first that $t \psi(1 / x^2) \ge C_3$. By~\revisionone{\eqref{eq:invest:2}},
\formula{
 t x \psi(1 / x^2) \sqrt{\psi^{-1}(1/t)} & = t \psi(1/x^2) \sqrt{\frac{\psi^{-1}(1/t)}{\psi^{-1}(\psi(1 / x^2))}} \ge \frac{t \psi(1 / x^2)}{(t \psi(1 / x^2))^{1 / \alpha}} \ge C_3^{1 - 1/\alpha} \ge 1 .
}
Hence, estimate~\eqref{eq:tauxest:reg:5} follows from~\eqref{eq:tauxest:reg:2} with arbitrary $\tilde{C}_1(\alpha) \le C_1$ and $\tilde{C}_2(\alpha) \ge 2 C_2$. Consider now the case $t \psi(1 / x^2) \le C_3$. Again by~\revisionone{\eqref{eq:invest:2}},
\formula{
 t x \psi(1 / x^2) \sqrt{\psi^{-1}(1/t)} & \le t \psi(1/x^2) \sqrt{\frac{\psi^{-1}(C_3/t)}{\psi^{-1}(\psi(1 / x^2))}} \le \frac{t \psi(1 / x^2)}{(t \psi(1 / x^2) / C_3)^{1 / \alpha}} \le C_3 .
}
Hence,
\formula{
 \pr(\tau_x > t) & \le 1 \le \frac{2C_3}{1 + t x \psi(1 / x^2) \sqrt{\psi^{-1}(1/t)}} \, .
}
Finally, by~\eqref{eq:tauxest:reg:2},
\formula{
 \pr(\tau_x > t) & \ge \pr(\tau_x > C_3 / \psi(1 / x^2)) \ge \frac{C_1}{C_3 x \sqrt{\psi^{-1}(\psi(1 / x^2) / C_3)}} \\
 & = \frac{C_1}{C_3} \sqrt{\frac{\psi^{-1}(1 / x^2)}{\psi^{-1}(\psi(1 / x^2) / C_3)}} \ge \frac{C_1}{C_3} \ge \frac{C_1}{C_3} \, \frac{1}{1 + t x \psi(1 / x^2) \sqrt{\psi^{-1}(1/t)}} \, .
}
Therefore, \eqref{eq:tauxest:reg:2} holds with arbitrary $\tilde{C}_1(\alpha) \le C_1 / C_3$ and $\tilde{C}_2(\alpha) \ge 2 C_3$.
\end{proof}

\begin{remark}
\label{rem:const}
\revisionone{From the proof of Theorem~\ref{th:tauxest:reg} it follows that the constants in this result are given by
\formula{
 C_1(\alpha, n) & = \frac{(\alpha - 1)^2 \gammal(n + \tfrac{1}{2}; (\alpha - 1)^2)}{4 \pi^2} \, , \\
 C_2(\alpha, n) & = \frac{60 (\gammal(n + 1 - \tfrac{1}{\alpha}; 1) + \gammau(n + \tfrac{1}{2}; 1))}{\pi^2 (\alpha - 1)} \, ,
}
and $C_3(\alpha, n) > 1$ is large enough, so that for $s \ge C_3(\alpha, n)$,
\formula{
 \tfrac{2}{\pi} s \, \gammau(n; (\alpha - 1)^2 s) & \le C_1(\alpha, n) .
}
In a similar way, in Corollary~\ref{cor:nzero},
\formula{
 \tilde{C}_1(\alpha) & = \frac{C_1(\alpha, 0)}{C_3(\alpha, 0)} , & \tilde{C}_2(\alpha) & = 2 C_2(\alpha, 0) + 2 C_3(\alpha, 0) .
}
}
\end{remark}

\begin{proof}[Proof of Theorem~\ref{th:tauxasymp:reg}]
\emph{Part~\ref{th:tauxasymp:reg:a}.}
\revisionone{As before, let $\beta = 2$.} We claim that by \revisiontwo{the Dominated Convergence Theorem}
,
\formula{
 & \lim_{x \to 0^+} (x \psi(1 / x^2) J_n(t, x, \tfrac{2}{x})) \\
 & \hspace*{3em} = \frac{2}{\pi \Gamma(\gamma) \tabs{\cos \tfrac{\gamma \pi}{2}}} \int_0^\infty (\cos \thet_\lambda)^2 e^{-t \psi(\lambda^2)} \lambda^2 (\psi'(\lambda^2))^2 (\psi(\lambda^2))^{n-1} \D \lambda
}
for all $n \ge 0$ and $t > 0$. Indeed, the left-hand side is the limit of integrals (see~\eqref{eq:ij}), with integrands convergent pointwise to the integrand in the right-hand side by Lemma~\ref{lem:flambdareg}. Furthermore, by Lemma~\ref{lem:flambdaest:reg}, the integrands in the left-hand side are bounded by
\formula{
 \frac{80}{\pi (\alpha - 1)} \cos \thet_\lambda e^{-t \psi(\lambda^2)} \lambda^2 (\psi'(\lambda^2))^2 (\psi(\lambda^2))^{n-1} ,
}
which is easily shown to be integrable in $\lambda \in (0, \infty)$, because $\lambda^2 \psi'(\lambda^2) \le \tfrac{\beta}{2} \psi(\lambda^2)$. The claim is proved.

On the other hand, by Lemma~\ref{lem:iest:reg}, for $x \in (0, 1)$,
\formula{
 x \psi(1 / x^2) \tabs{I_n(t, x, \tfrac{2}{x})} & \le \frac{2 x \psi(1 / x^2) \gammau(n; (\alpha - 1)^\beta t \psi(1 / x^2))}{\pi t^n} \le c(\alpha, \beta, n, t) x .
}

\noindent\emph{Part~\ref{th:tauxasymp:reg:b}.}
\revisionone{Again let $\beta = 2$.} Fix $x > 0$ and $a = \tfrac{2}{x}$. Observe that
\formula{
 & t^{n+1} \sqrt{\psi^{-1}(1 / t)} J_n(t, x, a) \\
 & \hspace*{1em} = \frac{4}{\pi} \, t^{n+1} \sqrt{\psi^{-1}(1 / t)} \int_0^a e^{-t \psi(\lambda^2)} \psi'(\lambda^2) (\psi(\lambda^2))^n \, \frac{F_\lambda(x)}{2 \lambda \psi'(\lambda^2) \cos \thet_\lambda} \, (\cos \thet_\lambda)^2 \, \frac{\lambda^2 \psi'(\lambda^2)}{\psi(\lambda^2)} \, \D \lambda
}
for all $n \ge 0$ and $t > 0$. By Lemmas~\ref{lem:fzero} and~\ref{lem:thetalambda:reg}, and Karamata's theorem~\cite[Theorem~1.5.11]{bib:bgt87},
\formula{
 \lim_{\lambda \to 0^+} \frac{F_\lambda(x)}{2 \lambda \psi'(\lambda^2) \cos \thet_\lambda} & = v(x) , & \lim_{\lambda \to 0^+} \thet_\lambda & = \frac{\pi}{\delta} - \frac{\pi}{2} , & \lim_{\lambda \to 0^+} \frac{\lambda^2 \psi'(\lambda^2)}{\psi(\lambda^2)} & = \frac{\delta}{2} \, .
}
We claim that
\formula{
 \lim_{t \to \infty} \expr{\frac{4}{\pi} \, t^{n+1} \sqrt{\psi^{-1}(1 / t)} e^{-t \psi(\lambda^2)} \psi'(\lambda^2) (\psi(\lambda^2))^n \ind_{(0, a)}(\lambda) \D \lambda} & = \frac{2 \Gamma(n + 1 - \tfrac{1}{\delta})}{\pi} \, \delta_0(\D \lambda) ,
}
with the vague limit of measures in the left-hand side. Indeed, the density function converges 
to $0$ \revisiontwo{uniformly} on $[\eps, a)$ for every $\eps > 0$. Furthermore, by a substitution $s = t \psi(\lambda^2)$,
\formula{
 & \lim_{t \to \infty} \expr{\frac{4}{\pi} \, t^{n+1} \sqrt{\psi^{-1}(1 / t)} \int_0^a e^{-t \psi(\lambda^2)} \psi'(\lambda^2) (\psi(\lambda^2))^n \D \lambda} \\
 & \hspace*{3em} = \lim_{t \to \infty} \expr{\frac{2}{\pi} \int_0^{t \psi(a^2)} \sqrt{\frac{\psi^{-1}(1 / t)}{\psi^{-1}(s / t)}} \, e^{-s} s^n \D s} = \frac{2}{\pi} \int_0^{\infty} e^{-s} s^{n - 1/\delta} \D s = \frac{2 \Gamma(n + 1 - \tfrac{1}{\delta})}{\pi} \, ;
}
the second equality follows by \revisiontwo{the Dominated Convergence Theorem}
, because $\psi^{-1}$ is regularly varying at zero with index $\tfrac{2}{\delta}$, and $\psi^{-1}(1 / t) / \psi^{-1}(s / t) \le \max(s^{-2 / \alpha}, s^{-2 / \beta})$ for $s, t > 0$ by \revisionone{Lemma}~\ref{l:ratioest} and~\revisionone{\eqref{eq:invest:2}}. The claim is proved.

It follows that
\formula{
 \lim_{t \to \infty} \expr{t^{n+1} \sqrt{\psi^{-1}(1 / t)} J_n(t, x, a)} & = \tfrac{2}{\pi} \Gamma(n + 1 - \tfrac{1}{\delta}) v(x) (\cos(\tfrac{\pi}{\delta} - \tfrac{\pi}{2}))^2 \tfrac{\delta}{2} .
}
Finally, by Lemma~\ref{lem:iest:reg},
\formula{
 t^{n+1} \sqrt{\psi^{-1}(1 / t)} \tabs{I_n(t, x, a)} & \le \tfrac{2}{\pi} t \sqrt{\psi^{-1}(1 / t)} \, \gammau(n; (\alpha - 1)^\beta t \psi(1 / x^2)) ,
}
and the right-hand side converges to $0$ as $t \to \infty$.
\end{proof}

\begin{remark}
\label{rem:lim}
\revisionone{From the proof Theorem~\ref{th:tauxasymp:reg} it follows that in part~\ref{th:tauxasymp:reg:a},
\formula{
 & \lim_{x \to 0} \expr{\tabs{x} \Psi(\tfrac{1}{\tabs{x}}) (-\tfrac{\D}{\D t})^n \pr(\tau_x > t)} \\
 & \hspace*{3em} = \frac{1}{2 \pi \Gamma(\gamma) \tabs{\cos \tfrac{\gamma \pi}{2}}} \int_0^\infty (\cos \thet_\lambda)^2 e^{-t \Psi(\lambda)} (\Psi'(\lambda))^2 (\Psi(\lambda))^{n-1} \D \lambda ,
}
with $\thet_\lambda$ given by~\eqref{eq:theta-def}. Also, in part~\ref{th:tauxasymp:reg:b},
\formula{
 \lim_{t \to \infty} \expr{t^{n + 1} \Psi^{-1}(\tfrac{1}{t}) (-\tfrac{\D}{\D t})^n \pr(\tau_x > t)} & = \frac{\delta \Gamma(n + 1 - \tfrac{1}{\delta}) (\sin \tfrac{\pi}{\delta})^2}{\pi} \, v(x) \revisiontwo{,}
}
for all $n \ge 0$ and $x \in \R \setminus \{0\}$, where $v(x)$ is the compensated potential kernel of~$X$.}
\end{remark}

\begin{remark}
\revisionone{The proofs clearly indicate that} the hypotheses of \revisionone{Theorem~\ref{th:tauxest:reg}} can be slightly relaxed to the following: $\psi(\xi^2)$ is the L\'evy--Khintchine exponent of a symmetric L\'evy process; $1 / (1 + \psi(\xi^2))$ is integrable; $(\psi_\lambda)_\lambda(\xi)$ is well-defined and $(\psi_\lambda)_\lambda(\xi)$ and $\xi / (\psi_\lambda)_\lambda(\xi)$ are increasing in $\xi > 0$ for all $\lambda > 0$; \revisionone{scaling-type condition}~\eqref{eq:reg:1} holds for some $\alpha \in (1, 2]$ and all $\xi > 0$, \revisionone{and a similar upper bound $\xi \Psi''(\xi) / \Psi'(\xi) \le \beta - 1$ holds for some $\beta \in (1, 2]$ and all $\xi > 0$} (the upper bound is now non-trivial also for $\beta = 2$). Apparently, these conditions can be further weakened at the price of more technical arguments. Since many important examples already belong to the class considered in this article, we decided to focus on simplicity rather than complete generality.
\end{remark}

%
%

\section*{Acknowledgments}

We thank the anonymous referees for careful reading of the manuscript and numerous comments, which helped us improve the article.

%
%

%
%


\begin{thebibliography}{00}

\bibitem{bib:bgt87}
N.~H.~Bingham, C.~M.~Goldie, J.~L.~Teugels,
\emph{Regular Variation}.
Cambridge University Press, Cambridge, 1987.

\bibitem{bib:bg64}
R.~M.~Blumenthal, R.~K.~Getoor,
\emph{Local times for Markov processes}.
Z. Wahrscheinlichkeitstheorie Verw. Gebiete 3 (1964): 50--74.

\bibitem{bib:bg68}
R.~M.~Blumenthal, R.~K.~Getoor,
\emph{Markov Processes and Potential Theory}.
Pure Appl. Math., Academic Press, New York 1968. 

\bibitem{bib:bgr61}
R.~M.~Blumenthal, R.~K.~Getoor, D.~B.~Ray,
\emph{On the distribution of first hits for the symmetric stable processes}.
Trans. Amer. Math. Soc. Vol. 99(3) (1961) 540--554.

\bibitem{bib:bgr14a}
K.~Bogdan, T.~Grzywny, M.~Ryznar,
\emph{Density and tails of unimodal convolution semigroups}.
J.~Funct. Anal. 266 (2014): 3543--3571.

\bibitem{bib:bgr14}
K.~Bogdan, T.~Grzywny, M.~Ryznar,
\emph{Barriers, exit time and survival probability for unimodal L\'evy processes}.
Probab. Theory Relat. Fields, to appear, DOI:10.1007/s00440-014-0568-6.

\bibitem{bib:c10}
F.~Cordero,
\emph{On the excursion theory for the symmetric stable L\'evy processes with index $\alpha \in {]1, 2]}$ and some applications}.
PhD thesis, Universit\'e Pierre et Marie Curie--Paris VI, 2010.

\bibitem{bib:gk68}
B.~V.~Gnedenko, A.~N.~Kolmogorov,
\emph{Limit distributions for sums of independent random variables}.
Addison--Wesley, 1968.

\bibitem{bib:g14}
T.~Grzywny,
\emph{On Harnack inequality and H\"older regularity for isotropic unimodal L\'evy processes}.
Potential Anal. 41(1) (2014): 1--29.

\bibitem{bib:k69}
H.~Kesten,
\emph{Hitting probabilities of single points for processes with stationary independent increments}.
Memoirs of the American Mathematical Society 93, American Mathematical Society, Providence, 1969.

\bibitem{bib:kkpw14}
A.~Kuznetsov, A.~Kyprianou, J.~C.~Pardo, A.~R.~Watson,
\emph{The hitting time of zero for a stable process}.
\revisionfinal{Electron. J. Probab. 19 (2014), 30:1--35.}

\bibitem{bib:k11}
M.~Kwa{\'s}nicki,
\emph{Spectral analysis of subordinate Brownian motions on the half-line}.
Studia Math. 206(3) (2011): 211--271.

\bibitem{bib:k12}
M.~Kwa\'snicki,
\emph{Spectral theory for one-dimensional symmetric L\'evy processes killed upon hitting the origin}.
Electron. J. Probab. 17 (2012), 83:1--29.

\bibitem{bib:k15}
M.~Kwa\'snicki,
\emph{Rogers functions and fluctuation theory}.
Preprint, 2013, arXiv:1312.1866.

\bibitem{bib:kmr12}
M.~Kwa{\'s}nicki, J.~Ma{\l}ecki, M.~Ryznar,
\emph{First passage times for subordinate Brownian motions}.
Stoch. Proc. Appl 123 (2013): 1820--1850.

\bibitem{bib:ls14}
J.~Letemplier, T.~Simon,
\emph{Unimodality of hitting times for stable processes}.
Preprint, 2013, arXiv:1309.5321v2.

\bibitem{bib:m13}
A.~Mimica,
\emph{Harnack Inequality and H\"older Regularity Estimates for a L\'evy Process with Small Jumps of High Intensity}.
J.~Theor.~Probab. 26(2) (2013): 329--348.

\bibitem{bib:m76}
D.~Monrad,
\emph{L\'evy processes: absolute continuity of hitting times for points}.
Z. Wahrscheinlichkeitstheorie Verw. Gebiete 37(1) (1976): 43--49.

\bibitem{bib:p14}
H.~Pant\'{\i},
\emph{On L\'evy processes conditioned to avoid zero}.
Preprint, 2013, arXiv:1304.3191.

\bibitem{bib:p08}
G.~Peskir.
\emph{The law of the hitting times to points by a stable L\'evy process with no negative jumps}.
Electron. Commun. Probab. 13 (2008):653--659.

\bibitem{bib:p67}
S.~C.~Port,
\emph{Hitting times and potentials for recurrent stable processes}.
J.~Anal. Math. 20(1) (1967) 371--395.

\bibitem{bib:r83}
L.C.G.~Rogers,
\emph{Wiener--Hopf factorization of diffusions and L\'evy processes}.
Proc. London Math. Soc. 47(3) (1983): 177--191.

\bibitem{bib:s99}
K.~Sato,
\emph{L{\'e}vy Processes and Infinitely Divisible Distributions}.
Cambridge Univ. Press, Cambridge, 1999.

\bibitem{bib:ssv10}
R.~Schilling, R.~Song, Z.~Vondra\v{c}ek,
\emph{Bernstein Functions: Theory and Applications}.
De Gruyter, Studies in Math. 37, Berlin, 2012.

\bibitem{bib:ssv06}
H.~\v{S}iki\'c, R.~Song, Z.~Vondra\v{c}ek,
\emph{Potential Theory of Geometric Stable Processes}.
Probab. Theory Relat. Fields 135 (2006): 547--575.

\bibitem{bib:s11}
T.~Simon,
\emph{Hitting densities for spectrally positive stable processes}.
Stochastics 83(2) (2011): 203--214.

\bibitem{bib:ss75}
K.~Soni, R.P.~Soni,
\emph{Slowly Varying Functions and Asymptotic Behavior of a Class of Integral Transforms I}.
J. Anal. Appl. 49 (1975): 166-179.

\bibitem{bib:y10}
K.~Yano,
\emph{Excursions away from a regular point for one-dimensional symmetric L\'evy processes without Gaussian part}.
Potential Anal. 32(4) (2010): 305--341.

\bibitem{bib:y12}
K.~Yano,
\emph{On harmonic function for the killed process upon hitting zero of asymmetric L\'evy processes}.
J. Math-for-Industry 5 (2013): 17--24.

\bibitem{bib:yyy09a}
K.~Yano, Y.~Yano, M.~Yor,
\emph{On the laws of first hitting times of points for one-dimensional symmetric stable L\'evy processes}.
In: \emph{S\'eminaire de Probabilit\'es XLII}, Lecture Notes in Math. 1979: 187--227, Springer, Berlin, 2009.

\bibitem{bib:yyy09b}
K.~Yano, Y.~Yano, M.~Yor,
\emph{Penalising symmetric stable L\'evy paths}.
J.~Math. Soc. Japan 61(3) (2009):757--798.

\end{thebibliography}
\end{document}